\documentclass[a4paper,11pt]{article}
\usepackage[utf8]{inputenc}

\usepackage[usenames,dvipsnames,svgnames,table]{xcolor} 
\usepackage{amssymb,amsfonts,amsrefs}
\usepackage{amsmath,amsthm}
\usepackage[active]{srcltx}
\usepackage{pdfsync}
\usepackage{graphicx}
\usepackage{hyperref}

\newcommand{\R}{\mathbb{R}}
\newcommand{\C}{\mathbb{C}}
\newcommand{\N}{\mathbb{N}}

\renewcommand{\P}{\mathbb{P}}
\renewcommand{\L}{\mathbb{L}}
\newcommand{\dd}{{\rm \,d}}
\newcommand{\FF}{\mathfrak{F}}
\DeclareMathOperator*{\esssup}{ess\,sup}

\newtheorem{theorem}{Theorem}[section]
\newtheorem{proposition}[theorem]{Proposition}
\newtheorem{lemma}[theorem]{Lemma}
\newtheorem{corollary}[theorem]{Corollary}

\theoremstyle{definition}
\newtheorem{definition}[theorem]{Definition}
\newtheorem{example}[theorem]{Example}
\theoremstyle{remark}
\newtheorem{remark}[theorem]{Remark}
\numberwithin{equation}{section}
\begin{document}

\title{Hexagonal structures in 2D Navier-Stokes flows}

\author{Lorenzo Brandolese\footnote{
Institut Camille Jordan,
Université de Lyon,
Université Lyon 1,
43 bd. du 11 Novembre
69622 Villeurbanne Cedex, France}}

\maketitle

\abstract{
Geometric structures naturally appear in fluid motions. One of the best known examples is \href{https://fr.wikipedia.org/wiki/Hexagone_de_Saturne#/media/Fichier:PIA20513_-_Basking_in_Light.jpg}{\emph{\color{blue}Saturn's Hexagon}}, the huge cloud pattern at the level of Saturn's north pole, remarkable both for the regularity of its shape and its stability during the past decades. In this paper we will address the spontaneous formation of hexagonal structures in planar viscous flows, in the classical setting of Leray's solutions of the Navier--Stokes equations. Our analysis also makes evidence of the isotropic character of the energy density of the fluid for sufficently localized 2D flows in the far field:  it implies, in particular, that fluid particles of such flows are nowhere at rest at large distances.}

\section{Introduction}

\label{sec:point}
%

We consider the 2D Navier--Stokes equations, 
\begin{equation}
 \left\{
 \begin{aligned}
 &\partial_t u+\nabla\cdot(u\otimes u)=\Delta u-\nabla p,\\
 &\nabla\cdot u=0\\
 &u(x,0)=u_0(x),
 \end{aligned}
 \right.
 \qquad x\in\R^2,\;t>0
 \tag{NS}
\end{equation}
where $u$ denotes the velocity field and $p$ the pressure.
The initial velocity $u_0$ is given. Throughout this paper, we will assume $u_0\in L^2_\sigma(\R^2)$, the space of $L^2$ and divergence-free vector-fields in $\R^2$.
In this case, it is well known that there exists a unique global Leray's weak solution, \emph{i.e.}, a solution
solving (NS) in the weak sense, such that
$u\in L^\infty((0,\infty),L^2_\sigma(\R^2))\cap L^2((0,\infty),\dot H^1(\R^2))$, and satisfying
the energy equality
\begin{equation}
 \|u(t)\|^2_2+2\int_0^t\|\nabla u(s)\|^2_2\dd s=\|u_0\|_2^2, \qquad t>0.
\end{equation}
This solution is also known to be in $C([0,\infty),L^2_\sigma(\R^2))$ and to solve integral equation \begin{equation}
\label{NSI}
\tag{NSI}
u(x,t)=e^{t\Delta}u_0-\int_0^t e^{(t-s)\Delta}\P\nabla\cdot(u\otimes u)(s)\dd s,
\qquad x\in\R^2,\quad t>0.
\end{equation}
Here, $\P$ is Leray's projector onto divergence-free vector fields and $e^{t\Delta}$ denotes the heat kernel.


The purpose of this paper is to show that, in the absence of any external forcing, and without of any special structure of the initial data, 
the flow  reveals regular geometric patterns in the far field. Our main results essentially are the following:

\begin{itemize}
\item[i)] 
Under \emph{mild decay} assumptions on $u_0$ and its derivatives at infinity, the euclidean norm of velocity variations, \emph{i.e.}, the quantity
$|u(\cdot,t)-u_0(\cdot)|$, tends to be constant, for a fixed $t>0$, in all the points of circles of large radii. In particular,
if $\zeta,\zeta'\in \mathbb{S}^1$, then
\[
|u(R\zeta,t)-u_0(R\zeta)|\sim|u(R\zeta',t)-u_0(R\zeta')|, \qquad\text{for $R\gg1$}.
\] 
See Theorem~\ref{th:mild} below.
Under \emph{stronger decay} assumptions on $u_0$ at infinity (and no decay condition on its derivatives), the 
speed of the fluid $|u(\cdot,t)|$ tends to be constant on circles of large radii:
\[
|u(R\zeta,t)|\sim |u(R\zeta',t)|, \qquad\text{for $R\gg1$}.
\]
This means that the energy density field, $x\mapsto\frac12|u(x,t)|^2$, is asymptotically radial at large distances. 
See Theorem~\ref{th:strong} below. 
A striking corollary is the following: 
\begin{center}
\emph{For generic flows, fluid particles are nowhere at rest at large distances,}
\end{center}
 in the sense that for all time $t>0$, for some $R_t>0$ and all
$|x|\ge R_t$, one has $|u(x,t)|\not=0$.
\item[ii)]
In the case of strongly decaying data, in contrast with the above isotropic behavior of the speed $|u(\cdot,t)|$, the components of the velocity
have a {genuinely anisotropic behavior} in the far field.
Namely, any component $v$ of the velocity field
spontaneously creates a \emph{rigid and regular hexagonal structure}.
More precisely, for any fixed $t>0$, there are exactly six exceptional
directions along which the decay of $v(x,t)$ as $|x|\to\infty$ is faster. Such curious structures appear immediately and, after a time-dependent rescaling,
rigidly rotate during the evolution, without changing their shape. See Figure~\ref{fig1} and Theorem~\ref{th:hexagonal} below.
We will also estimate the angular speed of such structures and show that for generic solutions their angular speed goes to zero as $t\to+\infty$: moreover, when the initial data belong to $\dot H^{-1}(\R^2)$ these structures converge for long time to a stationary position. See Corollary~\ref{cor:larget}.
This corollary will reveal an unexpected geometric interpretation of the classical energy dissipation problem of Leray's solution for large time.  
In the case of initial data with just a mild decay the picture described above is 
sightly different: the hexagonal structures 
appear for each component of the vector field $x\mapsto u(x,t)-u_0(x)$.

\end{itemize}

Our strategy will be to associate to any Leray's solution
a complex valued map 
$z\colon \R\to \C$, 
defined by formula~\eqref{complex} below,
such that $|z(t)|$ is independent on the chosen coordinate system, encoding the most important asymptotic properties of the solution in the far field.
The main results are Theorem~\ref{th:mild} and Theorem~\ref{th:hexagonal},
the latter being probably more surprising.

In this paper we focused on 2D \emph{finite energy} flows, with possibly \emph{poorly localized vorticities}. 
In a companion paper, \cite{Bra-Vor}, we discuss the case of planar flows with 
\emph{well localized vortity}, but possibly \emph{infinite energy}. The two papers are thus complementary. In~\cite{Bra-Vor} we also compute the spatial asymptotics as $|x|\to\infty$ at any order: such higher-order asymptotics reveal more general polygonal structures.

\begin{figure}
\label{fig1}
\hskip-1cm
\includegraphics[scale=0.30]{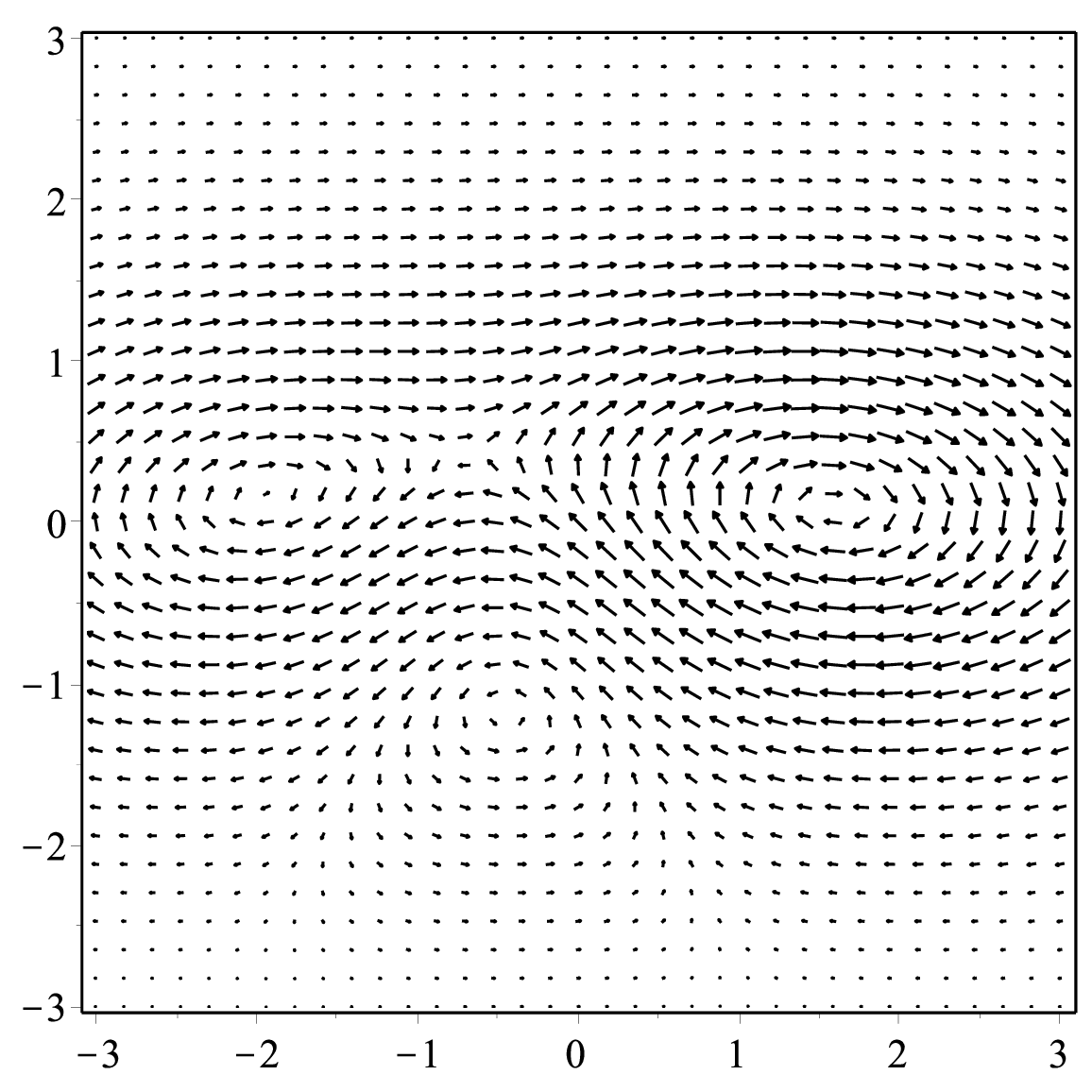}\hfill
\includegraphics[scale=0.30]{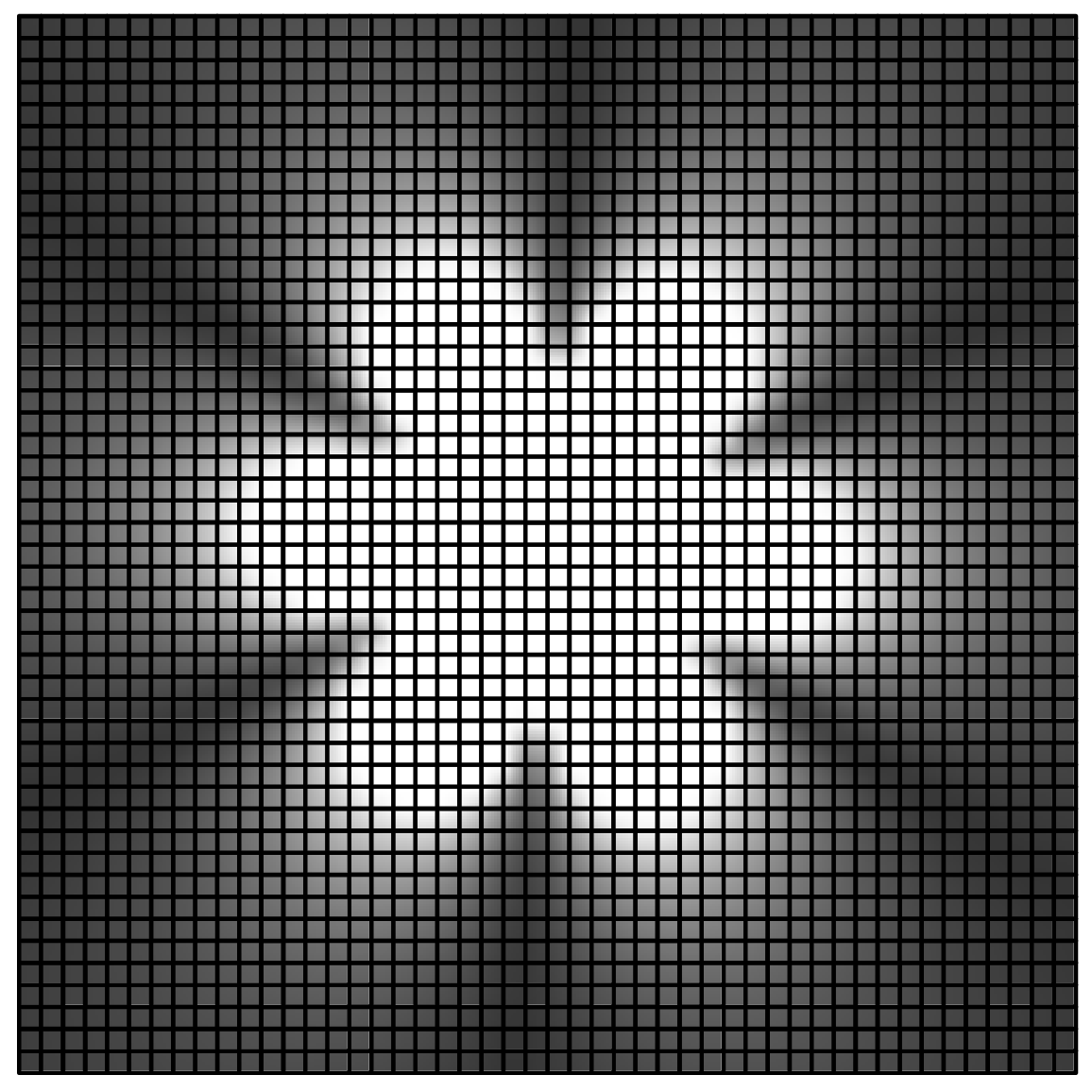}\hfill
\includegraphics[scale=0.30]{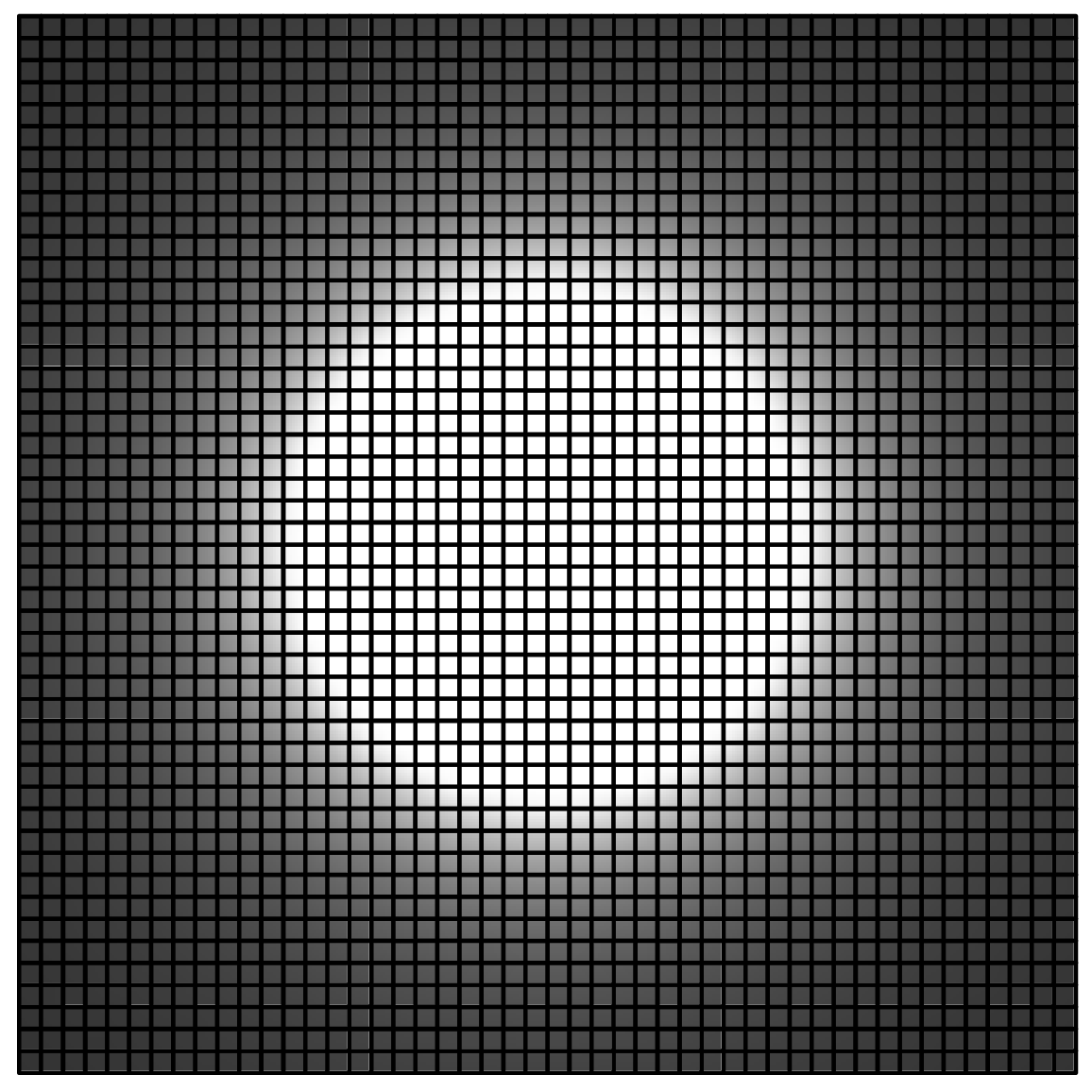}
\caption{\small
Left: a well-localized initial datum with no special symmetry.
Middle: density plot of the horizontal speed $|u_1(\cdot,t)|$, obtained neglecting the lower-order terms at the spatial infinity, in the limit as $t\to0$: the neighbourhoods of the six directions of faster decay are darker. This hexagonal structure, suitably rescaled, is present for any time $0<t<\infty$ and rigidly rotates during the evolution.
If lower-order terms are not neglected, then larger and larger scales are needed to detect this structure as $t$ approaches~$0$ or as~$t$ grows to infinity.
The shape of the structure is essentially the same, excepted for a change in scale and orientation,  for any generic well localized datum and any component of the velocity field. 
Right: in contrast with the speed of the individual components of the velocity field, the
energy density $x\mapsto\frac12|u(x,t)|^2$ is asymptotically radial.}
\end{figure}

\section{Statement of the main results}

A more formal statement of our first assertion in the Introduction
is provided by Theorem~\ref{th:mild} below.
We denote by $\mathbb{S}^1$ the unit circle centered at the origin.

\begin{theorem}
\label{th:mild}
Let $u_0\in L^2_\sigma \cap L^p(\R^2)$ for some $p>4$ and 
\begin{equation}
\label{dec-u0}
\begin{cases}
 u_0(x)=o(|x|^{-1}\log(|x|)^{-1/2})\\
 \nabla u_0(x) = o(|x|^{-2}\log(|x|)^{-1/2})\\
 \Delta u_0(x)=o(|x|^{-3}).
\end{cases}
\text{as $|x|\to\infty$}
\end{equation}
Let $u$ be the unique global Leray's solution starting from $u_0$.
Then, for all $t>0$, the   limit
\begin{equation}
\label{def:L}
 L(t)=\lim_{|x|\to+\infty} |x|^3|u(x,t)-u_0(x)|
\end{equation}
does exist and is given by
\begin{equation}
\label{limit:L}
L(t):=\textstyle\frac1\pi\sqrt{\bigl(\textstyle\int_0^t\!\!\int (u_1^2-u_2^2)\dd y\dd s\bigr)^2
+\bigl(\int_0^t\!\!\int 2u_1u_2\dd y \dd s\bigr)^2}\,.
\end{equation}
In particular, for all $\zeta\in\mathbb{S}^1$, the radial limits
\begin{equation}
\label{def:L2}
\lim_{R\to+\infty} R^3|u(R\zeta,t)-u_0(R\zeta)|
\end{equation}
do exist and are \emph{independent} on~$\zeta$.
\end{theorem}

The decay assumptions~\eqref{dec-u0} are natural for finite energy flows: the first of~\eqref{dec-u0} is nothing but a pointwise analogue of the usual condition $u_0\in L^2(\R^2)$; moreover,  often $\nabla u_0$ and
$\Delta u_0$ decay respectively one and two decay rates faster than $u_0$ (unless $u_0$ has an oscillating behavior at infinity), which is more than needed in the second and third condition of~\eqref{dec-u0}.

We now give a formal statement of what can be inferred in the case of faster decaying data. In this case, no decay of higher-order derivatives is needed. 
\begin{theorem}
 \label{th:strong}
 Let $u_0\in L^2_\sigma\cap L^p(\R^2)$ for some $p>2$.
  Assume also that $u_0(x)=o(|x|^{-3})$ as $|x|\to\infty$.
 Then the unique  global Leray's solution~$u$ starting from $u_0$ satisfies, for all $t>0$,
 \begin{equation}
\label{def:L3}
  L(t)=\lim_{|x|\to+\infty} |x|^3|u(x,t)|
 = \lim_{R\to+\infty} R^3|u(R\zeta,t)| \quad\text{(independent on $\zeta\in\mathbb{S}^1$)},
\end{equation}
where $L(t)$ is given by formula~\eqref{limit:L}.
\end{theorem}
Generically, in both~\eqref{def:L} and \eqref{def:L3}, the limit $L(t)$ will be nonzero, which means that $|u(\cdot,t)|$ is asymptotically radial at the spatial infinity.

We emphasise that, generically, the limits 
$\lim_{|x|\to\infty} |x|^3\bigl(u(x,t)-u_0(x)\bigr)$
(in the setting of Theorem~\ref{th:mild})
and $\lim_{|x|\to\infty} |x|^3 u(x,t)$ 
(in the setting of Theorem~\ref{th:strong}) \emph{do not exist}. So our results put in evidence a property
of the speed (or of the energy) of fluid particles, rather than of their velocity vectors.

The statement of Theorem~\ref{th:mild} might look less attractive than that of
Theorem~\ref{th:strong}. However, the former is a deeper result.
Indeed, the decay condition, of the latter, namely $u_0(x)=o(|x|^{-3})$, is too stringent to be physically realistic, because such a strong decay condition is known to immediately break down during the evolution. (See~\cite{DobS94}).
 On the other hand, the milder decay conditions
on $u_0$ required in Theorem~\ref{th:mild} are preserved by the Navier--Stokes flow.
Hence, the decay conditions of Theorem~\ref{th:mild} will be satisfied in many physically relevant cases.

Formula~\eqref{limit:L}
suggests the introduction of the complex valued map
\begin{equation}
\label{complex}
z(t)=
\Bigl(\int_0^t\!\!\int (u_1^2-u_2^2)\dd y\dd s\Bigr)
+i\Bigl(\int_0^t\!\!\int 2u_1u_2\dd y\dd s\Bigr),
\end{equation}
so that
\[
L(t)=\textstyle\frac{1}{\pi}|z(t)|.
\]
The use of integrals of the form  $\int u_ju_k$ is inspired by earlier
papers by Schonbek~\cite{Sch91}, Dobrokhothov and Shafarevich \cite{DobS94} and Miyakawa and Schonbek~\cite{MiyS01}.
Choosing an appropriate coordinate system (suitably rotating the axis by a time dependent angle), one can always set
$\int_0^t \!\!\int(u_1^2-u_2^2)=0$, or otherwise $\int_0^t\!\!\int 2u_1u_2=0$,
thus simplifying the expression of $L(t)$ in~\eqref{limit:L}. Formula~\eqref{limit:L}
has however the advantage of being independent on the choice of the coordinate system. 

Next theorem shows that, contrary to $|u(\cdot,t)|$ or $|u(\cdot,t)-u_0|$,
the individual \emph{components} of the velocity field have a genuinely anisotropic behavior.
We will use the notation $A(x)\approx B(x)$ to indicate that the ratio $(A/B)(x)$
converges to a non-zero real constant. 
For an arbitrarily fixed unit vector $\mathbf{e}$ of $\R^2$, let us denote
$v=u\cdot{\bf e}$ the component of $u$ along ${\bf e}$. In the same way,
we denote $v_0=u_0\cdot{\bf e}$.

\begin{theorem}
\label{th:hexagonal}
  Let $u_0\in L^2_\sigma(\R^2)$ satisfying the conditions as in Theorem~\ref{th:mild}, or otherwise as in
  Theorem~\ref{th:strong}.
  Let $v$ be any component of Leray's solution~$u$.
  For any $t>0$ such that $L(t)\not=0$,
there exists $\zeta_t\in \C$, with $|\zeta_t|=1$, such that
the regular hexagon 
\[
\mathcal{H}=\mathcal{H}(t)
=\bigl\{\zeta_1,\zeta_2,\ldots,\zeta_6\bigl\},
\]
made of the complex roots of the equation $z^6=\zeta_t$,
has the following property:
as $R\to+\infty$,
\begin{itemize}
\item[i)]  in the setting of Theorem~\ref{th:mild},
   \[
   \begin{aligned}
   &v(x,t)-v_0(x)\approx |x|^{-3} 
    \qquad&\text{for $x=R\zeta$, $|\zeta|=1$, $\zeta\not\in \mathcal{H}$,}\\
   &v(x,t)-v_0(x)=o(|x|^{-3})
    \qquad&\text{for $x=R\zeta$, $|\zeta|=1$, $\zeta\in \mathcal{H}$}
   \end{aligned}
    \]
\item[ii)] or, simply, in the setting of Theorem~\ref{th:strong},
 \[
   \begin{aligned}
   &v(x,t)\approx |x|^{-3} 
    \qquad&\text{for $x=R\zeta$, $|\zeta|=1$, $\zeta\not\in \mathcal{H}$,}\\
   &v(x,t)=o(|x|^{-3})
    \qquad&\text{for $x=R\zeta$, $|\zeta|=1$, $\zeta\in \mathcal{H}$.}
   \end{aligned}
    \] 

\end{itemize}
\end{theorem}

Figure~\ref{fig1} (middle) offers a possible visualization of this theorem. 
A different way to visualize the conclusion of Theorem~\ref{th:hexagonal} is to
perform an inverse stereographic projection and to draw the images of the level lines of $|v(t)-v_0|$ on the stereographic sphere: our theorem predicts that close the north stereographic pole the level lines tend to have a snowflake shape with a hexagonal symmetry.
Yet, Theorem~\ref{th:hexagonal} cannot be used to explain the physical phenomenon of Saturn's hexagon mentioned in the abstract. We refer to~\cite{PNAS-2020} for a recent analysis of the latter. 
Close to Saturn's north pole the hexagonal pattern is due to an hexagonal symmetry of the trajectories of fluid particles. The hexagonal structure in our paper appears as a symmetry of the absolute values of the components of $u(x, t)$ (for large $x$), but it is not a symmetry of the fluid velocity itself.

\begin{example}
If, in  a given coordinate system, $v=u_2$ is the vertical component of Leray's solution,
then the  proof will show that one can take $\zeta_t=(|z(t)|^2/z(t)^2)$, which is well defined because of the assumption $L(t)\not=0$. Here, $z(t)$ is given by~\eqref{complex}.
For the horizontal component, $v=u_1$, one has $\zeta_t=(-|z(t)|^2/z(t)^2)$.
So the two hexagonal structures associated with the horizontal and vertical components of $u$ are always obtained from each-other performing a rotation of $\pm\pi/6$.
\end{example}

The proof of the above theorems relies on a refinement of an asymptotic formula for Navier--Stokes flows in $\R^d$, $d\ge2$, 
\begin{equation}
\label{as-pro}
u(x,t)=e^{t\Delta}u_0(x)+\underbrace{\nabla H(x,t)}_{\approx|x|^{-d-1}}+o(|x|^{-d-1}),
\qquad\text{as $|x|\to+\infty$},
\end{equation}
first established in~\cite{BraV07} under appropriate decay assumptions
on $u_0$. (See also Lemarié-Rieusset's book~\cite[Theorem 4.12]{Lem16} and~\cite{KukR11}):
See~\eqref{def:Ha} below for the definition of the scalar function~$H$.
The crucial observation of the present paper is that, focusing on the two-dimensional case, we can put in evidence two very important properties specific of planar fluids that, surprisingly, remained unnoticed in earlier studies. The first one is that, when $d=2$, the map $x\mapsto|\nabla H(x,t)|$ is a radial function. The second one is that each components of $\nabla H(\cdot,t)$ possess exactly six zeros on the unit circle, that are the vertex of a (time-dependent) regular hexagon.
Let us emphasize that the explicit expression of $\nabla H(\cdot,t)$ reduces the proof
of both properties to a short and elementary computation. This computation is contained in Section~\ref{sec:proofth}.

Concerning the technical contributions, in this paper we perform a new asymptotic analysis of the nonlinear term in (NSI), that will allow us to considerably relax the required conditions on $u_0$ to insure the validity~\eqref{as-pro}. Indeed, in earlier papers the conditions on~$u_0$ were too stringent to encompass the case of only mild decaying data as in~\eqref{dec-u0}.
This will require a deeper use of the cancellations hidden inside the Oseen kernel,
whereas simpler size estimates were enough in the more restrictive setting 
considered in~\cites{BraV07,Lem16,KukR11}.
Another simple but useful ingredient will be an asymptotic formula for the solutions of the heat equation, that
allows us to replace $e^{t\Delta}u_0(x)$ with $u_0(x)$ in~\eqref{as-pro}.

\begin{remark}
\label{rem:compa}
As mentioned in the introduction, 
in a companion paper, \cite{Bra-Vor}, we discuss the case of possibly \emph{infinite energy flows} with localized vorticity $\omega=\partial_1u_2-\partial_2u_1$.
Therein, we show in this case that the dominant geometric feature in the far field of the components of $u$ is no longer hexagonal: a density plot like that in Figure~\ref{fig1} (middle) would reveal a \emph{digonal symmetry} when the total circulation of the flow $\int\omega_0$ is non-zero and a \emph{quadrilobe shape} for flows with zero total circulation.
The approach of~\cite{Bra-Vor} is mainly based on the Biot--Savart law, and it is better suited for the geometric description of the higher-order terms. Moreover, it allows to recover in a simple way a few of the essential features of the recent analysis of McOwen, Sultan and Topalov \cites{McOT19, McOT21,SulT20} on the Euler equations in asymptotic spaces.

On the other hand, the advantage of the approach of 
the present paper is that it allows us to encompass  (according to \eqref{dec-u0}) the case of flows with non-integrable vorticity, for which the total circulation is not even defined. This will be especially relevant in Section~\ref{large_time}, when we will discuss the large time behavior of Leray's solutions.
\end{remark}

\section{The homogeneous part of the kernel of $e^{t\Delta}\P{\rm div}.$}

\subsection{Decompositions of the kernel in $\R^d$, $d\ge2$.}

It is convenient to denote by $F(x,t)=(F_{j,h,k})(x,t)$
the kernel of the operator $e^{t\Delta}\P\text{div\,}$. Equivalently, $F$ can be defined
through its symbol $(j,h,k=1,2)$, with $\xi\in \R^d$ and $t>0$,
\begin{equation}
\label{eq:symb-F}
\widehat{F}_{j,h,k}(\xi,t)=e^{-t|\xi|^2}
\Bigl(i\xi_h\delta_{j,k}+\frac{i\xi_j i\xi_h i\xi_k}{|\xi|^2}\Bigr),
\end{equation}
where $\delta_{j,k}$ is the Kronecker symbol.
Therefore, the integral formulation of the Navier--Stokes equations in $\R^d$, $d\ge2$ reads
\begin{equation}
\begin{split}
u_j(t)&=e^{t\Delta} u_{0,j}-\int_0^t\!\!\int \sum_{h,k=1}^d F_{j,h,k}(x-y,t-s)(u_hu_k)(s)\dd y\dd s,
\qquad(j=1,\ldots,d),\\
\end{split}
\end{equation}
together with the incompressibility condition $\nabla\cdot u_0=0$. In more compact form, we will write
\begin{equation}
\label{IE}
u(t)=e^{t\Delta} u_0-\int_0^t F(t-s)*(u\otimes u)(s)\dd s,
\qquad
\nabla\cdot u_0=0.
\end{equation}

Let us first recall some known properties of the kernel $F$.
First of all, we have the scaling relation
\begin{equation}
 \label{scaF}
 F(x,t)=t^{-(d+1)/2}F(x/\sqrt t,1).
\end{equation}
Computing the inverse Fourier transform in~\eqref{eq:symb-F}, using the identity $|\xi|^{-2}=\int_0^\infty e^{-s|\xi|^2}\dd s$, yields the usual decomposition $F=F^{(1)}+F^{(2)}$,
with
\begin{equation}
\label{usdecF}
F^{(1)}_{j;h,k}(x,t)=\partial_h g_t\,\delta_{j,k}
\qquad
F^{(2)}_{j;h,k}(x,t)=\int_t^\infty \partial_j\partial_h\partial_k
g_s(x)\dd s.
\end{equation}
Here, $g_t(x)=(4\pi t )^{-d/2}e^{-|x|^2/(4t)}$ is the heat kernel.
From this decomposition it is easy to deduce the classical pointwise bound
\begin{equation}
 \label{podeF}
 \sup_{x\in\R^d,t>0}{|x|^{d+1}|F(x,t)|}<\infty.
\end{equation}

There is another decomposition of the kernel $F$, pointed out in~\cite{BraARMA}, holds:
it reads
\begin{equation}
\label{proARMA}
F(x,t)= \FF(x)
 +|x|^{-d-1}\Psi\Bigl(x/\sqrt t\Bigr).
\end{equation}
Here $\FF$ is a homogeneous tensor of degree $-d-1$, whose components are given by
\begin{equation}
\label{FF}
\begin{split}
\FF_{j;h,k}(x)&=
\partial_j\partial_h\partial_k E_d(x)\\
   \end{split}
\end{equation}
where $E_d$ is the fundamental solution of the Laplacian
in~$\R^d$.
Moreover, 
$\Psi=(\Psi_{j,h,k})$ is smooth outside the origin and such that, for all $\alpha\in\N^d$, and $x\not=0$,
there exist $C,c>0$ such that  
\[
|\partial^\alpha \Psi(x)|\le Ce^{-c|x|^2}.
\]

This second decomposition of $F$ is very useful in the study of the far-field asymptotics of the velocity field. Indeed, in formula~\eqref{as-pro}, the scalar function $H$ is given by
\begin{equation}
\label{def:Ha}
\begin{split}
H(x,t)
&:=\Bigl[\nabla^2 E_d(x)\colon\int_0^t\!\!\int (u\otimes u)(y,s)\dd y \dd s\Bigr]\\
&=\sum_{h,k=1}^d\partial_h\partial_k E_d(x)\int_0^t\!\!\int (u_hu_k)(y,s)\dd y\dd s.
\end{split}
\end{equation}
Therefore, the vector field $\nabla H$ is constructed by taking linear combinations of components of the tensor $\FF$.

\section{Persistence results on pointwise decay}

The goal of this section is to state a couple of propositions about the persistence of pointwise decay
for $u_0$ and its derivatives. The former is needed to establish Theorem~\ref{th:mild} and the second
to establish Theorem~\ref{th:strong}. Results in this vein go back to Takahashi~\cite{Tak99} and were refined by several authors, see~\cites{Miy02, Vig05}.
However, the precise assertions of the propositions below do not seem to be covered by earlier results.


\begin{proposition}
\label{prop:dec12}
 Let $u_0\in L^2_\sigma(\R^2)\cap L^p(\R^2)$ for some $4<p\le \infty$, be such that
\begin{equation}
\label{dec-u1}
\begin{cases}
 u_0(x)=o(|x|^{-1}(\log|x|)^{-1/2})\\
 \nabla u_0(x) = o(|x|^{-2}(\log|x|)^{-1/2})
 \end{cases}
 \qquad\text{as $|x|\to\infty$.}
 \end{equation}
Then the unique Leray's solution~$u$ starting from~$u_0$ satisfies, for all $0<T<\infty$,
\begin{equation}
\label{dec-u2}
\begin{cases}
 \sup_{t\in(0,T)} t^{1/p} |u(x,t)|=o(|x|^{-1}(\log|x|)^{-1/2})\\
\sup_{t\in(0,T)} t^{1/2+1/p}|\nabla u(x,t)|  = o(|x|^{-2}(\log|x|)^{-1/2})
 \end{cases}
  \qquad\text{as $|x|\to\infty$.}
 \end{equation}
\end{proposition}

We need the techical assumption that $u_0\in L^p(\R^2)$ with $p>4$ to deduce, from standard heat kernel estimates and a fixed point argument, that $\nabla(u\otimes u )\in L^1_{\rm loc}(\R^+,L^\infty(\R^2))$. This information is useful in proving~Proposition~\ref{prop:dec12}.

\begin{proposition}
\label{prop:dec32}
 Let $u_0\in L^2_\sigma\cap L^p(\R^2)$, for some $p>2$, be such that
\begin{equation}
\label{dec-ua}
 u_0(x)=o(|x|^{-3/2})
 \qquad\text{as $|x|\to\infty$.}
 \end{equation}
Then the unique Leray's solution~$u$ starting from~$u_0$ satisfies, for all  $0<T<\infty$,
\begin{equation}
\label{dec-uta}
 \sup_{t\in(0,T)}t^{1/p}\,|u(x,t)|=o(|x|^{-3/2})
  \qquad\text{as $|x|\to\infty$.}
 \end{equation}
\end{proposition}

These results would remain true with more general decay profiles. See the comments after the proof.
The main role of the additional $L^p$-assumptions is to prevent a too singular behavior of the solution near $t=0$.
The proof of these propositions is postponed in Section~\ref{sec:pro}.

\section{Asymptotics of the nonlinear term}

The main issue of this section is the spatial asymptotics of the linear integral term
\begin{equation}
\label{lin-int}
 \L(w)(x,t)=\int_0^t\!\!\int F(x-y,t-s)w(y,s)\dd y\dd s.
\end{equation}
As the analysis of~\eqref{lin-int} is independent on the space dimension, in this section we work in~$\R^d$, $d\ge2$.
We have in view the application of the results of this section to the quadratic term
\begin{equation*}
 w(x,t)=(u\otimes u)(x,t).
\end{equation*}
To this purpose, let us establish two lemmas.

\begin{lemma}
\label{lem:cruc2}
Let $0<T<\infty$, $w\in (L^1(\R^d\times(0,T)))^{d\times d}$ and $0\le a<1$.
Assume that
\[
\esssup_{t\in(0,T)}t^{a}|\nabla w(x,t)| = o(|x|^{-d-1}\log(|x|)^{-1}),
\]
as $|x|\to\infty$.
Then, for $|x|\not=0$ and $t\in(0,T)$,
\begin{equation}
\label{concl:lem}
 \L(w)(x,t)=\FF(x)\colon\int_0^t\!\!\int w(y,s)\dd y\dd s+|x|^{-d-1}\epsilon(x,t),
\end{equation}
where
\[
\,|\epsilon(x,t)|\to0\text{\; as $|x|\to\infty$}.
\]
\end{lemma}
The colon symbol stands for a summation on the last two subscripts of
$\FF_{j,h,k}(x)$, as in~\eqref{def:Ha}.
Let us now state our second Lemma:

\begin{lemma}
 \label{lem:cruc}
Let $0<T\le\infty$, $w\in L^1(\R^d\times(0,T))$ and $0\le a<1$.
Assume that
\[
\esssup_{t\in(0,T)}t^a|w(x,t)| = o(|x|^{-d-1}), \qquad \text{as $|x|\to\infty$}.
\]
Then conclusion~\eqref{concl:lem} holds.

\end{lemma}

The first Lemma is the most interesting one, as its conclusion is reached dropping any decay assumption on $w$: it just relies on a condition on $\nabla w$, that is usually less stringent than the corresponding decay condition on~$w$ itself, at least when $w$ is the quadratic nonlinearity of the Navier--Stokes equations.
As we will see, the proof the latter lemma is elementary. On the other hand
the proof of the former makes use of deeper cancellation properties of the kernel $F$.

\begin{proof}[Proof of Lemma~\ref{lem:cruc2}.]
Let us decompose
\[
\L(w)=(\L_1+\L_2+\L_3)(w),
\]
with
\begin{equation}
\begin{split}
\label{dec L} 
\L_1(w)(x,t)&:=\int_0^t\!\!\int_{|y|\le |x|/2}F(x-y,t-s)w(y,s)\dd y \dd s, \\
\L_2(w)(x,t)&:=\int_0^t\!\!\int_{|y|\le |x|/2}F(y,t-s)w(x-y,s)\dd y\dd s, \quad\text{and}\\
\L_3(w)(x,t)&:=\int_0^t\!\!\int_{|x-y|\ge |x|/2,\,\,|y|\ge |x|/2}F(x-y,t-s)w(y,s)\,dy\dd s.\\
\end{split}
\end{equation}
We dropped the colon symbol between $F$ and $w$ to simplify the notations and proceed as all the functions were scalar.
We start estimating $\L_2(w)$.

Then we have
\begin{equation}
\label{dec-gra}
|\nabla w(x,t)|\le Ct^{-a}|x|^{-d-1}(\log(e+|x|))^{-1}\epsilon_2(x),
\end{equation}
for some constant $C>0$ independent on $x$ and $t\in(0,T)$, and a function $\epsilon_2$, independent on time, such that
 $\epsilon_2(x)\to0$ as $|x|\to\infty$.
A crucial observation is that
\begin{equation}
 \int_{|y|\le R} F(y,t)\dd y=0, \qquad\text{for all $t>0$ and $R>0$},
\end{equation}
as one easily checks applying~\eqref{usdecF} and the antisymmetries of $F^{(1)}$ and
$F^{(2)}$.
Therefore, we can rewrite $\L_2(w)$ as
\[
\L_2(w)(x,t)=\int_0^t\!\!\int_{|y|\le |x|/2} F(y,t-s)[w(x-y,s)-w(x,s)]\dd y\dd s.
\]
Applying the gradient estimate~\eqref{dec-gra} we get
\begin{equation*}
 |\L_2(x,t)|\le C|x|^{-d-1}(\log(e+|x|))^{-1}\int_0^t\!\!\int_{|y|\le |x|/2}
  |F(y,t-s)|\,|y|\, s^{-a} \dd y  \dd s \,
 \Bigl(\sup_{|y|\ge|x|/2}\epsilon_2(y)\Bigr).
\end{equation*}
But, from~\eqref{scaF} and~\eqref{podeF}, we see that
\[
|F(y,t-s)|\le C \min\{|y|^{-d-1},(t-s)^{-(d+1)/2}\}.
\]
Hence,
\[
\begin{split}
\Bigl| \int_0^t\!\!\int_{|y|\le |x|/2}  &|F(y,t-s)|\,|y|\,s^{-a}\dd y \dd s\Bigr|\\
&\le C\int_0^t\int_{|y|\le \sqrt{t-s}} (t-s)^{-(d+1)/2}|y|\,s^{-a}\dd y\dd s 
    +C\int_0^t\int_{\sqrt{t-s}\le |y|\le |x|/2} |y|^{-d}s^{-a}\dd y\dd s\\
&\le C\Bigl(t^{1-a}
 +{\bf 1}_{|x|\ge2\sqrt{t}}\int_0^t\log(|x|/(2\sqrt{t-s}))s^{-a}\dd s\Bigr)\\
&\le CT^{1-a}\log(e+T)\log(e+|x|).
\end{split}
\]
We conclude that, for $\tilde\epsilon_1(x)=\sup_{|y|\ge|x|/2}\epsilon_1(y)$,
\[
\begin{split}
\sup_{t\in(0,T)}|\L_2(x,t)| 
&=\,o(|x|^{-d-1}).
\end{split}
\]

We next estimate $\L_3(w)$.
Using~\eqref{podeF} we get
\[
|\L_3(x,t)|\le |x|^{-d-1}\int_0^t\!\!\int_{|y|\ge |x|/2} |w(y,s)|\dd y\dd s.
\]
As $w\in L^1(\R^d\times(0,T))$, by the dominated convergence theorem
 $\int_0^t\!\int_{|y|\ge |x|/2} |w(y,s)|\dd y\dd s\to0$ as $|x|\to\infty$, uniformly with respect to $t\in(0,T)$. 


We end up with the analysis of $\L_1(w)$.
Recalling~\eqref{proARMA}, we split $\L_1(w)$
as
\begin{equation}
\label{de-L1}
\begin{split}
	\L_1(w)(x,t)=& \FF(x)\colon\!\!\int_0^t \!\!\int w(y,s)\,dy\,ds\\
	&\quad\quad - \FF(x)\colon\!\!\int_0^t \!\!\int_{|y|\ge |x|/2} w(y,s)\,dy\,ds\\
	&\quad\quad+\int_0^t\!\!\int_{|y|\le |x|/2}\bigl[F(x-y,t-s)-F(x,t-s)\bigr]\colon w(y,s)\,dy\,ds\\
  &\quad\quad +|x|^{-d-1}\int_0^t  \Psi(x/\sqrt{t-s})\colon\!\!\int_{|y|\le |x|/2} w(y,s)\,dy\,ds.
\end{split}
\end{equation}
For the last term in~\eqref{de-L1}, we can make use, e.g.,
of the rough estimate
$|\Psi(y)|\le C|y|^{-2(1-a)}$,
that implies the bound, for this last term,
\[
 Ct^{1-a}\,|x|^{-d-3+2a}\, \|w\|_{L^1(\R^d\times(0,T))}
\]
which decays faster than $|x|^{-d-1}$ as $|x|\to\infty$.
Hence the last term in~\eqref{de-L1} is settled.

Now, let us consider the third term in the right-hand side of~\eqref{de-L1}.
It is well known, and easy to check with~\eqref{proARMA}, that 
$|\nabla F(x,t)|\le C|x|^{-d-2}$.
Therefore, the third term in~\eqref{de-L1} is bounded by
\[
C|x|^{-d-2}\int_0^t\!\!\int_{|y|\le |x|/2} |y|\,|w(y,s)|\dd y\dd s.
\]
By the dominated convergence theorem,
\[
\int_0^T\!\!\int |x|^{-1}|y|\,|w(y,s)|\,{\bf1}_{|y|\le |x|/2}(y)\dd y\dd s\to0
\quad \text{as $|x|\to\infty$}.
\]
Therefore, the third term in the right-hand side of~\eqref{de-L1} is $o(|x|^{-d-1})$ 
as $|x|\to\infty$, uniformly in $t\in(0,T)$. 
This settles also the third term~\eqref{de-L1}

The second term in the right-hand side of~\eqref{de-L1} is the simplest one, and can be treated as 
$\L_3$.
Summarising,
we proved that
\[
\L(w)(x,t)= \FF(x)\colon\int_0^t \!\!\int w(y,s)\,dy\,ds +|x|^{-d-1}\epsilon(x,t),\\
\]
where
\[
|\epsilon(x,t)|\le (1+t^{1-a}\log(e+t))\tilde\epsilon(x)
\]
with $\tilde\epsilon$ independent on $t$ and such that $\tilde \epsilon(x)\to0$
as $|x|\to+\infty$.

\end{proof}

\begin{proof}[Proof of Lemma~\ref{lem:cruc}]
Going back to the decomposition~\eqref{dec L}  of~$\L(w)$, we see that $\L_1(w)$ and $\L_3(w)$ can be treated exactly as before.
The estimate of $\L_2(w)$ is more direct:
indeed, by the assumption on $w$, there exists $C>0$ independent on $x$ and $t$,
and a function $\epsilon_2$ independent on $t$, with $\epsilon_2(x)\to0$, as $|x|\to\infty$, 
such that
\begin{equation}
\label{dec-w}
|w(x,t)|\le C\,t^{-a}|x|^{-d-1}\epsilon_2(x).
\end{equation}
Then we have
\[
\begin{split}
 |\L_2(w)|(x,t) 
 &\le C|x|^{-d-1}\int_0^t\!\!\int_{|y|\le |x|/2} |F(y,t-s)| s^{-a}\epsilon_1(x-y)\dd y\dd s\\
 &\le C|x|^{-d-1}\int_0^t \|F(t-s)\|_1 s^{-a}\dd s \sup_{|y|\ge |x|/2}\epsilon_1(y)\\
 &\le C|x|^{-d-1} t^{1/2-a}\sup_{|y|\ge |x|/2}\epsilon_1(y).
 \end{split}
\]

Hence $\L_2(w)(x,t)=o(|x|^{-d-1})$ as $|x|\to\infty$, for all fixed $t\in(0,T)$.
Notice that in conclusion~\eqref{concl:lem} we have now
\[
|\epsilon(x,t)|\le t^{1/2-a}\tilde \epsilon(x)
\]
with $\tilde\epsilon$ independent on~$t$ and such that $\tilde\epsilon(x)\to0$
as $|x|\to\infty$.
\end{proof}

\section{Linear asymptotics in 2D}

In this section we put in evidence what conditions on $u_0$ ensure that
 \[
 \sup_{t\in(0,T)}|e^{t\Delta}u_0(x)|=o(|x|^{-3}), \qquad\text{ as $|x|\to\infty$}
 \]
We denote by $g_t(x)=(4\pi t)^{-1}e^{-|x|^2/4t}$ the 2D heat kernel, for $t>0$ and $x\in\R^2$.

\begin{lemma}
\label{lem:h}
Let $u_0\in L^2(\R^2)$. Assume also that at least one of the following conditions holds:
\begin{itemize}
\item[(i)] Either $u_0(x)=o(|x|^{-3})$, or
\item[(ii)] $\nabla u_0(x)=o(|x|^{-3})$, or
\item[(iii)] $\Delta u_0(x)=o(|x|^{-3})$ .
 \end{itemize}
 Then there exists a polynomial
 $P=P(T)$ and a function $\epsilon=\epsilon(x)$, such that
 $\lim_{|x|\to\infty}\epsilon(x)=0$ and
\begin{equation}
 \label{esti:h}
  \sup_{t\in(0,T)}|e^{t\Delta}u_0(x)-u_0(x)|\le P(T)|x|^{-3}\epsilon(x)
\end{equation}
for all $T>0$.
\end{lemma}

\begin{proof}
For $m=0,1,2$, let us introduce the four terms
\begin{equation*}
\begin{split}
&D_1\equiv\int_{|y|\le |x|/2} \Bigl[u_0(x-y)-\sum_{|\gamma|\le m-1} 
\frac{(-1)^{|\gamma|}}{\gamma!}\partial^\gamma u_0(x)y^\gamma\Bigr]g_t(y)\dd y,\\
&D_2\equiv\int_{|y|\le |x|/2} g_t(x-y)  u_0(y)\dd y,\\
&D_3\equiv\int_{|y|\ge |x|/2,\; |x-y|\ge |x|/2} u_0(x-y)g_t(y)\dd y\dd y
\end{split}
\end{equation*}
and
\[
D_4\equiv - \sum_{|\gamma|\le m-1} 
\frac{(-1)^{|\gamma|}}{\gamma!}\partial^\gamma u_0(x)
\int_{|y|\ge|x|/2} y^\gamma g_t(y)\dd y.
\]
In the case (i), we choose above $m=0$, so that 
\[
e^{t\Delta}u_0=D_1+D_2+D_3\qquad\text{and}\qquad D_4=0.
\]
Using that $\|g_t\|_1=1$, we see that the $D_1$ and  $D_3$ integrals are $o(|x|^{-3})$, uniformly with respect to $t\in(0,\infty)$. For $D_2$, we use $\sup_{|y|\le|x|/2} g_t(x-y)\le Ct|x|^{-4}$, and that $\int_{|y|\le |x|/2}|u_0(y)|\dd y\le C\log(e+|x|)$.
Here, and thoughout the proof, $C>0$ will denote a suitable constant depending only on $u_0$.
This proves the result~\eqref{esti:h} with a polynomial $P$ of degree~$1$.

In the case (ii), we choose $m=1$ above, so that  $D_4=-u_0(x)\int_{|y|\ge |x|/2} g(y)\dd y$.
 As $\int g_t=1$, we see that 
\[
e^{t\Delta}u_0-u_0=D_1+D_2+D_3+D_4.
\]
 By the first-order Taylor formula and assumption~(ii),
 \[
 |D_1|\le \bigl(\sup_{|y|\ge |x|/2}|\nabla u_0(y)|\bigr)\int_{|y|\le |x|/2}|y|g_t(y)\dd y
 =\sqrt t\, o(|x|^{-3}).
 \] This estimate for $D_1$ is in agreement with~\eqref{esti:h}.  
For $D_2$ we can use the inequality 
\[
\sup_{|y|\le|x|/2} g_t(x-y)\le C|x|^{-7}t^{5/2}.
\]
Next, the gradient estimate of $u_0$ implies that $u_0$ is Lipschitz outside a ball of large radius. Hence, $\int_{|y|\le |x|/2} |u_0|\le C(1+|x|)^3$. The fast decay of the heat kernel thus settles the $D_2$ integral.
For the integral $D_3$, using again the Lipschitz property of $u_0$ we have
\[
|D_3|\le C\int_{|y|\ge |x|/2}(1+|y|)g_t(y)\dd y\le C(t^2+t^{5/2})|x|^{-4}.
\] 
For $D_4$, we write, for some $R_0>0$ dependent only on $u_0$ and all $|x|\ge R_0$,
\[
|D_4|\le C|x|\int_{|y|\ge |x|/2}|y|g_t(y)\dd y\le Ct^{3}|x|^{-4}.
\] 
This establishes~\eqref{esti:h} with a polynomial~$P$ of degree~3.

In case (iii), 
we choose $m=2$. Notice that in the summations over $\gamma$, all the termes 
 corresponding to $|\gamma|=1$ vanish after integrating with respect to $y$ (because
of the anti-symmety of $y_1g_t(y)$ and $y_2g(y)$. Hence, $D_4$ is the same as in case (ii) and $D_1+D_2+D_3+D_4$ equals $e^{t\Delta}u_0-u_0$, as before. But now in $D_1$ we can apply the second-order Taylor formula.
In fact, in $D_1$, the mixed derivatives $\partial_j\partial_k u_0$ will play no role when $j\not= k$, because 
$y_jy_k g(y)$ is anti-symmetric and vanish after integration.
And $\int_{|y|\le |x|/2} y_1^2g_t(y)\dd y=\int_{|y|\le|x|/2} y_2^2g_t(y)\dd y$.
Hence,
\[
D_1=\frac14\int_0^1\int_{|y|\le|x|/2} \Delta u_0(x-\theta y) |y|^2g_t(y)\dd y\dd \theta.
\]
Thus,
\[|D_1|\le C\,t\,\bigl(\sup_{|y|\ge |x|/2} |\Delta u_0(y)|\bigr),\]
and this can be bounded as in
the right-hand side of~\eqref{esti:h}, by assumption (iii). 
To estimate the other terms, we first need a control on the growth of $u_0$ at infinity.
We can write $u_0=\phi+\chi u_0$ where $\phi$ is an $L^2$-compactly supported function,
$\chi$ is smooth and $\chi\equiv0$ near the origin, 
$\chi\equiv1$ in a neighbourhood of infinity, and such that $|\Delta(\chi u_0)(x)|\le C(1+|x|) ^{-3}$.
Letting 
\[
\psi=\frac1{2\pi}\int \log(|x-y|)\Delta(\chi u_0)(y)\dd y,
\] we see that
$\Delta\psi=\Delta( \chi u_0)$, so that $\chi u_0-\psi$ is a harmonic polynomial.
Moreover, $\psi$ has a logarithmic growth. It follows that $|u_0(x)|$ is bounded by some polynomial $Q(x)$ for large enough $|x|$.
But then, the estimates of the other terms $D_2$, $D_3$ and $D_4$ can be performed essentially as before. In~\eqref{esti:h}, the degree of $P$ will then 
depend on that of~$Q$.
\end{proof}

Notice that conclusion of the Lemma remains true if one replaces the $L^2$-condition on $u_0$ by the more general one $u_0\in \mathcal{E}'(\R^2)+L^1_{\rm loc}(\R^2)$,
where $\mathcal{E}'(\R^2)$ is the space of compactly supported distributions.
 In case iii), however, one would need also some ``controlled growth at infinity'' for $u_0$. For example, 
$u_0(x)=O(e^{|x|^\alpha})$ for some $0\le \alpha<2$.

\section{Proof of Theorem~\ref{th:mild}, 
Theorem~\ref{th:strong} and Theorem~\ref{th:hexagonal}}
\label{sec:proofth}

The proof of the main theorems is a simple consequence of the persistence properties of the spatial decay (Proposition~\ref{prop:dec12} and Proposition~\ref{prop:dec32}), of the two previous lemmas, and a few remarkable properties of the kernel $F$ in 2D.

\begin{proof}[Proof of Theorem~\ref{th:mild}]
Under the assumptions of Theorem~\ref{th:mild},
Proposition~\ref{prop:dec12} applies. Hence, $u$ satisfies~\eqref{dec-u2}.
Hence, $w=u\otimes u$ does satisfy the conditions of Lemma~\ref{lem:cruc2} 
with $a=1/2+2/p<1$ and $d=2$.

Hence, applying also case (iii) of Lemma~\ref{lem:h}, we get, for all fixed $t\in(0,T)$,

\begin{equation}
\label{as:pro1}
\begin{split}
 u(x,t)
 &=u_0(x)+\FF(x):\int_{0}^t\!\int (u\otimes u)(y,s)\dd y\dd s + o_t(|x|^{-3})\\
 &=u_0(x)+\nabla H(x,t)+ o_t(|x|^{-3}),
 \end{split}
\end{equation}
where $H$ was defined in~\eqref{def:Ha}.
Here, $o_t(|x|^{-3})$ denotes a time-dependent function decaying faster than $|x|^{-3}$ at the spatial infinity.

Let us now study in more detail the vector fields  of the form
$(x_1,x_2)\mapsto \nabla H(x_1,x_2,t)$.
Such vector fields of potential type  consist of homogeneous functions of degree $-3$: their components are linear combinations of third-order derivatives of $E_2(x_1,x_2)=-\frac{1}{4\pi}\log(x_1^2+x_2^2)$.
To this purpose, let us fix $t>0$ and denote
\begin{equation}
\label{en:mat}
\begin{pmatrix}
a&b\\b&d
\end{pmatrix}
:=
\begin{pmatrix}
\int_0^t\int u_1^2(x,s)\dd y\dd s & \int_0^t\int (2u_1u_2)(x,s)\dd y\dd s\\
\int_0^t \int (2u_1u_2)(x,s)\dd y\dd s &\int_0^t\int u_2^2(x,s)\dd y\dd s
\end{pmatrix}.
\end{equation}
From the expression of $E_2(x_1,x_2)$ we get
\begin{equation}
\label{nabH}
\begin{split}
 \nabla H(x_1,x_2,t)
 &=
\begin{pmatrix}
 \partial_1\\
 \partial_2
\end{pmatrix}
H(x_1,x_2,t)
=
\begin{pmatrix}
 a\partial_1^3+b\partial_1^2\partial_2+d\partial_1\partial_2^2\\
 a\partial_1^2\partial_2+b\partial_1\partial_2^2+d\partial_2^3
\end{pmatrix}
E_2(x_1,x_2)\\
&=
\frac{1}{\pi(x_1^2+x_2^2)^3}
\begin{pmatrix}
 {-(a-d)(x_1^3-3x_1x_2^2)-b(3x_1^2x_2-x_2^3)}\\
  {(a-d)(x_2^3-3x_1^2x_2)+b(x_1^3-3x_1x_2^2)}
\end{pmatrix}.
\end{split}
\end{equation}
Let us compute $|\nabla H(x,t)|=\sqrt{\partial_1 H(x,t)^2+\partial_2 H(x,t)^2}$.
A crucial remark, specific to the 2D case, is that $|\nabla H(x,t)|$ is a radial function for all possible choice of~$a$, $b$ and~$d$. Indeed, by a direct computation we get
\begin{equation}
\label{eq:speed}
\begin{split}
|\nabla H(x,t)|
&=\frac{\sqrt{(a-d)^2+b^2}}{\pi|x|^3}.
\end{split}
\end{equation}
It then follows that it does exist the limit
\[
\begin{split}
\lim_{|x|\to+\infty}|x|^{3}|u(x,t)-u_0(x)|=\textstyle\frac{1}{\pi}\sqrt{(a-d)^2+b^2}.
\end{split}
\]
Going back to the original notations, the above limit equals
\begin{equation*}
L(t)=\textstyle\frac1\pi\sqrt{\bigl(\textstyle\int_0^t\!\!\int (u_1^2-u_2^2)\dd y\dd s\bigr)^2
+\bigl(\int_0^t\!\!\int 2u_1u_2\dd y \dd s\bigr)^2}\,.
\end{equation*}
\end{proof}

\begin{proof}[Proof of Theorem~\ref{th:strong}]

The proof is the same as above, the only change is that one needs to apply 
Proposition~\ref{prop:dec32}
instead of Proposition~\ref{prop:dec12}, next case (i) of Lemma~\ref{lem:h}, and finally
and Lemma~\ref{lem:cruc}
instead of Lemma~\ref{lem:cruc2}.
\end{proof}

\begin{proof}[Proof of Theorem~\ref{th:hexagonal}.]
Because of the invariance of the Navier--Stokes equations under rotations, we can assume
without loss of generality that $v=u_2$, the vertical component
of the velocity field.
Under the assumptions of Theorem~\ref{th:mild},
by the asymptotic profile~\eqref{as:pro1},
we see that
\[
u_2(x,t)-u_{0,2}(x)=\partial_2 H(x,t)+o_t(|x|^{-3}) \qquad\text{as $|x|\to+\infty$}.
\]
Under the assumptions of Theorem~\ref{th:strong} the term $u_{0,2}(x)$ on the left-hand side
can be incorporated inside the remainder terms.

Let $P(\theta,t)=\nabla H(\cos\theta,\sin\theta,t)$.  We easily get, rewriting~\eqref{nabH} in terms of trigonometric functions,
\[
\begin{split}
P(\theta,t)
&=\frac{1}{\pi}
\begin{pmatrix}
 (d-a)\cos(3\theta)-b\sin(3\theta)\\
 (d-a)\sin(3\theta)+b\cos(3\theta)
\end{pmatrix}\\
&=\frac{\sqrt{(d-a)^2+b^2}}{\pi}
\begin{pmatrix}
 \cos(3\theta+\alpha)\\
 \sin(3\theta+\alpha)
\end{pmatrix},
\end{split}
\]
where the angle $\alpha\in[0,2\pi)$, 
which is the argument of the complex number $z(t)$,
is uniquely defined by the system
\begin{equation}
\label{angle}
\begin{cases}
\cos\alpha=\frac{d-a}{\sqrt{(d-a)^2+b^2}}\\
\sin \alpha=\frac{b}{\sqrt{(d-a)^2+b^2}}.
\end{cases}
\end{equation}
Here $\alpha$, just like $a$, $b$ and $d$, depends on time.
In particular,
\[
|P(\theta,t)|=\frac{\sqrt{(a-d)^2+b^2}}{\pi},
\]
and the fact that the right-hand side is independent of $\theta$ is another way of recovering the already observed fact that $|\nabla H(\cdot,t)|$ is radial. 
Moreover, for $x=(R\cos\theta,R\sin\theta)$, with $R>0$,
\[
\partial_2 H(x,t)=\partial_2 H(R\cos\theta,R\sin\theta)=
\frac{\sqrt{(d-a)^2+b^2}}{\pi}R^{-3}\sin(3\theta+\alpha).
\]
Then, for any fixed $t>0$ and $\theta\in[0,2\pi)$,
\[
R^3 u_2(R\cos\theta,R\sin\theta,t)-R^3u_{0,2}(R\cos\theta,R\sin\theta)=
\frac{\sqrt{(a-d)^2+b^2}}{\pi}\sin(3\theta+\alpha)+o(1),
\]
as $R\to+\infty$.
Let $t$ such that $L(t)\not=0$. For such times~$t$, we have $(a-d)^2+b^2\not=0$. 
It then just remains to check whether or not the term $\sin(3\theta+\alpha)$ vanishes. 
To the six zeros in $[0,2\pi[$ of the periodic 
function $\theta\mapsto\sin(3\theta+\alpha)$
correspond six distinct points in the circle: $e^{-i\alpha/3 +k\pi}$, $k=0,\ldots,5$. These are the $6^{\text{th}}$-complex roots of $e^{-2i\alpha}$. 
But $e^{i\alpha}=z(t)/|z(t)|$,
so the assertion of the theorem applies, for the vertical component $v=u_2$,
with~$\zeta_t=(|z(t)|^2/z(t)^2)$.
\end{proof}

The function $L(t)$ (see~\eqref{limit:L} for the definition) and the hexagons $\mathcal{H}(t)$ can be defined for any 2D Leray solution, even though one should not expect such objects play any special role, if one just assumes $u_0\in L^2_\sigma(\R^2)$, without any additional decay condition on the data.

Ruling out the non-generic situation in which the matrix  in~\eqref{en:mat} is a multiple of the identity matrix, \emph{i.e.}, assuming that 
\[
(d-a)^2+b^2\not=0,
\]
we see that the term $\nabla H(\cdot,t)$ is not identically zero. 
In fact one generically expects that $L(t)\not=0$.
This, of course, is an useful information in the application of asymptotic profiles like~\eqref{as-pro}.
Next remark allows to establish rigorously that $L(t)\not=0$ at least for a short time interval, as soon as one starts from a ``non-symmetric'' initial datum.
Therefore, the formation of hexagonal structures can be granted at least in some
interval $(0,T_0)$.
 
 \begin{remark}
 \label{cor:ns}
  Let $u_0\in L^2_\sigma(\R^2)$ and $u$ the associated Leray's solution. 
  Assume also that $u_0$ satisfies the following ``non-symmetry'' conditions:
 \begin{equation}
 \label{nonsym}
 \mbox{The $2\times 2$ matrix $\int (u_0\otimes u_0)(x)\dd x$ is \underline{not} a scalar multiple of the identity matrix.}
 \end{equation}
Then there exists $T_0>0$ such that
$L(t)\not=0$ for all $t\in(0,T_0)$.
\end{remark}

The proof of the remark is immediate: it relies on the fact that
\[
\lim_{t\to0^+}\frac{L(t)}{t}=
\textstyle\frac{1}{\pi}\sqrt{\bigl(\textstyle\int (u_{0,1}^2-u_{0,2}^2)\dd y\bigr)^2
+\bigl(\textstyle\int 2u_{0,1}u_{0,2}\dd y\bigr)^2.}
\]
Therefore, under condition~\eqref{nonsym}, $L(t)$ cannot vanish
when $t>0$ is small enough.

Notice that the non-symmetry condition~\eqref{nonsym} can be reformulated in an
equivalent way as follows: ``there exists a coordinate system such that $\int u_{0,1}^2\not=\int u_{0,2}^2$''. Yet another equivalent formulation is: ``there exists a coordinate system such that $\int u_{0,1}u_{0,2}\not=0$''.

\section{Large time behavior of hexagonal structures}
\label{large_time}

The goal of this section is to estimate the angular velocity 
\[
\dot{\mathcal{H}}(t)
\]
of the hexagons $\mathcal{H}(t)$. Even though the orientation of the hexagon $\mathcal{H}(t)$ depends on the of the component $v$ of the velocity field vector, its angular velocity is independent on~$v$.
To study the large time behavior of the spatial limit $L(t)$,
we need the following Lemma.

\begin{lemma}
\label{lem:largeti}
Let $u_0\in L^2_\sigma(\R^2)\cap \dot H^{-1}(\R^2)$. Then the corresponding Leray's solution belongs to $L^2([0,\infty),L^2(\R^2))$.
In this case, $L(t)$ does have a limit as $t\to+\infty$ and
\[
\lim_{t\to+\infty} L(t)=
\textstyle\frac1\pi\sqrt{\bigl(\textstyle\int_0^\infty\!\!\int (u_1^2-u_2^2)\dd y\dd s\bigr)^2
+\bigl(\int_0^\infty\!\!\int 2u_1u_2\dd y \dd s\bigr)^2}\,.
\]
\end{lemma}
\begin{proof}
Indeed, we have the obvious estimate $\|e^{t\Delta}u_0\|_2^2\le \|u_0\|_2^2$, and also
$\int e^{-2t|\xi|^2}|\widehat u_0(x)|^2\dd\xi\le \|u_0\|_{\dot H^{-1}}^2(\sup_{\xi}e^{-2t|\xi|^2}|\xi|^2)\lesssim \|u_0\|_{\dot H^{-1}}^2t^{-1}$.
Hence,
\[
\|e^{t\Delta}u_0\|_2^2\le C(1+t)^{-1}.
\]
Wiegner's theorem~\cite{Wie87} applies and gives the following 
$L^2$-estimate for the difference $u-e^{t\Delta}u_0$:
\[
\|u(t)-e^{t\Delta}u_0\|_2^2\le C (1+t)^{-2}\log^2(e+t).
\]
Therefore, $u\in L^2(\R^+,L^2(\R^2))$ if (and only if) $e^{t\Delta}u_0\in L^2(\R^+,L^2(\R^2))$.
But,
\[
\int_0^\infty\!\!\!\int e^{-2t|\xi|^2}|\widehat u_0(\xi)|^2\dd\xi\dd t=\int|\xi|^{-2}|\widehat u_0(\xi)|^2\dd \xi= \|u_0\|_{\dot H^{-1}}.
\]
Thus, $e^{t\Delta}u_0\in L^2(\R^+,L^2(\R^2))$ if and only if 
$u_0\in \dot H^{-1}(\R^2)$ and the conclusion follows.
\end{proof}

In view of our next corollary, let us introduce the following notion.
\begin{definition}
We call \emph{generic} a Leray's solution in $\R^2$ such that
\begin{equation}
\label{generic}
\underline{L}:=\liminf_{t\to+\infty} L(t)>0,
\end{equation}
where $L(t)$ is given by~\eqref{limit:L}.
\end{definition}

We do not attempt to give a precise topological description of this notion of genericity. This terminology is justified by the fact that the condition
$\lim_{t\to+\infty} L(t)=0$ is expected to achievable only with special solutions,
like those featuring specific simmetries, and that in all the other cases~\eqref{generic} holds.
Of course, in the case $u_0\in L^2_\sigma\cap \dot H^{-1}(\R^2)$, 
applying Lemma~\ref{lem:largeti}, shows that a Leray solution is generic if and only if
\begin{equation}
\label{non-symt}
\Bigl(\textstyle\int_0^\infty\!\!\int u_h^2u_k^2\dd y\dd s\Bigr)_{h,k} 
\text{
\emph{is not a scalar multiple of the identity matrix}}.
\end{equation} 
This is the analogue, for the solution $u$, of the non-symmetry condition~\eqref{nonsym} for the datum $u_0$.

Condition~\eqref{non-symt} first appeared in~\cite{MiyS01} in connection with 
the construction of \emph{fast dissipative flows}. Namely,
the main result of~\cite{MiyS01} essentially states that weak solutions of solution $u$ of Navier--Stokes in $\R^d$ are rapidly dissipative, \emph{i.e.} $\|u(t)\|_2^2=o(t^{-(d+1)/2})$ if and only if $\|e^{t\Delta}u_0\|_2^2=o(t^{-(d+1)/2})$ and $u$ 
does \underline{not} satisfy~\eqref{non-symt}.
See also~\cite{GalW02} for an insightful analysis of such flows using the invariant manifolds theory.

Next corollary reveals that condition~\eqref{non-symt},
and its more general formulation~\eqref{generic}, not only appers in the setting
of rapidly dissipative flows, but has a deeper signification in the large time
structure of the flow.

\begin{corollary}\hfill
\label{cor:larget}
\begin{itemize}
\item[i)]
For generic Leray's solutions, the angular speed $\dot{\mathcal{H}}(t)$ of the hexagonal structure 
is such that
\[
|\dot{\mathcal{H}}(t)|=O(\|u(t)\|^2_2)
\qquad\text{as $t\to+\infty$}.
\]
In particular, this angular speed slow down to zero for large time.
\item[ii)]
If $u_0\in \dot H^{-1}(\R^2)$ and $u$ satisfies~\eqref{non-symt}, then the hexagon $\mathcal{H}(t)$ converge to a stationary position $\mathcal{H}_{\infty}$ as $t\to+\infty$.
\end{itemize}
\end{corollary}

\begin{proof}
To estimate the angular speed of the hexagons 
$\mathcal{H}=\mathcal{H}(t)$ we compute the time
derivative of the function $\alpha=\alpha(t)$ defined in~\eqref{angle}.
We find
\begin{equation*}
\label{der-angle}
|\dot{\mathcal{H}}|=\frac{|b'(d-a)-b(d-a)'|}{(d-a)^2+b^2}.
\end{equation*}
Therefore, recalling the definition of $L=L(t)$ in~\eqref{limit:L},
\begin{equation*}
\begin{split}
|\dot{\mathcal{H}}|
&\le \frac{|b'|+|d'-a'|}{\sqrt{(d-a)^2+b^2} }\\
&=\frac{|\int 2u_1u_2\dd x|+|\int (u_1^2-u_2^2)\dd x|}{\pi L}\\
\end{split}
\end{equation*}
Maximizing the numerator under energy constraint we finally get
that the angular speed of $\mathcal{H}$ is estimated by
\begin{equation}
\label{ang-speed}
|\dot{\mathcal{H}}(t)|\le \frac{\sqrt2\,\|u(t)\|_2^2}{\pi L(t)}.
\end{equation}
Then the first conclusion follows from condition~\eqref{generic}.
By a classical result of Kato and Masuda, $\|u(t)\|_2^2\to0$ (see
\cite{Wie87} for a proof) and so
$|\dot{\mathcal{H}}(t)|\to0$ for generic solutions.
In the case $u_0\in L^2_\sigma\cap\dot H^{-1}(\R^2)$,  and \eqref{non-symt} holds,
the application of Lemma~\ref{lem:largeti} and an integration in time in an interval of the
form $[t_0,\infty)$ 
yields the second conclusion.
 
\end{proof}

 Conditions \eqref{generic} and~\eqref{non-symt}
 can be difficult to check for an arbitrarily given $u_0$.
 However, if the size of $u_0\in L^2_\sigma\cap \dot H^{-1}(\R^2)$ is small enough in the $L^2$-norm, 
 then such conditions are both very easily checked, using the Fourier transform, applying the following criterion.

\begin{proposition}
\label{prop:generic}
Let $u_0\in L^2_\sigma\cap \dot H^{-1}(\R^2)$,
be such that $\tilde u_0:=(-\Delta)^{-1/2}u_0$ is non-symmetric (in the sense that 
of condition~\eqref{nonsym} holds with $\tilde  u_0$ instead of $u_0$).
Then there exists $\delta>0$ such that if $\|u_0\|_2<\delta$, then
Leray's solution starting from $u_0$ does satisfy~\eqref{non-symt}.
\end{proposition}

\begin{proof}
Our condition that $\tilde u_0$ is non symmetric can be expressed
by the fact that
\[
\kappa_0:=
\sqrt{
\Bigl(\int \frac{|\widehat u_{0,1}(\xi)|^2-|\widehat u_{0,2}(\xi)|^2}{|\xi|^2}\dd \xi\Bigr)^2 
+ \Bigl(\int 2\frac{\widehat u_{0,1}(\xi)\overline{\widehat u_{0,2}}(\xi)}{|\xi|^2}\dd\xi\Bigr)^2}\not=0.
\]

Let us introduce the Banach space $X$ of measurable functions in $\R^2\times(0,\infty)$ such that
\[
\|u\|_X=\esssup_{t>0}\|u(t)\|_2+\esssup_{t>0}\sqrt t\,\|u(t)\|_\infty
+\Bigl(\int_0^\infty\|u(t)\|_2^2\dd t\Bigr)^{1/2}.
\]

If $u_0\in L^2_\sigma\cap \dot H^{-1}(\R^2)$ then we have by standard heat kernel estimates, and recalling the last line of the proof of Lemma~\ref{lem:largeti},
\[
\|e^{t\Delta}u_0\|_X\le C_0(\|u_0\|_2+\|u_0\|_{\dot H^{-1}}),
\]
where $C_0>0$ is an absolute constant.
To prove the bilinear estimate
\begin{equation}
 \label{bil:V}
 \|B(u,v)\|_X\le K\|u\|_X\|v\|_X,
\end{equation}
with $K$ independent on $u$ and $v$, 
we only have to
establish that
\[
\bigl(\int_0^\infty\|B(u,v)\|_2^2\dd t\Bigr)^{1/2} \le K' \|u\|_X\|v\|_X,
\]
as the other contributions of the $X$-norm of $B(u,v)$ are just standard Kato's estimates.

To establish the latter estimate, first  observe that if $u$ and $v$ belong to $X$, then 
$f:=\|u\|_2\|v\|_2\in L^1\cap L^\infty(\R^+)$, with norm bounded by $\|u\|_X\|v\|_X$.
Then,
\[
\begin{split}
\|B(u,v)\|_2(t) 
&\le \int_0^{t/2}\|F(t-s)\|_2 f(s)\dd s
  + \int_{t/2}^t \|F(t-s)\|_{6/5}\|u(s)\|_3\|v(s)\|_3\dd s\\
&\le Ct^{-1}\int_0^{t/2}f(s)\dd s + C\|u\|_X^{1/3}\|v\|_X^{1/3}\int_{t/2}^t (t-s)^{-2/3} f(s)^{2/3}  
 s^{-1/3}\dd s.  
  \end{split}
\]
We have $t^{-1}\int_0^{t/2}f(s)\dd s\le C(1+t)^{-1}\|u\|_X\|v\|_X$, which indeed 
in $L^2(\R^+)$ as a function of the $t$~variable.

To estimate the second term we will make use of
classical H\"older and Young inequalities for Lorentz spaces and their interpolation properties, see \cite{Lem02}. 
As $f^{2/3}\in L^{3/2}\cap L^\infty(\R^+)$ and  the map $s\mapsto s^{-1/3}$
belong to the weak-$L^{3}(\R^+)$ space, we get that the map $s\mapsto f(s)^{2/3}s^{-1/3}$ belongs the Lorentz space $L^{p,q}(\R^+)$, for all $1<p<3$ and $1\le q\le \infty$.
In particular, this map belongs to $L^{6/5,1}(\R^+)$.
On the other hand, the map $s\mapsto (t-s)^{-2/3}$ belongs to the 
weak-$L^{3/2}(\R^+)$ space and 
$L^{3/2,\infty}*L^{6/5,1}(\R^+)\subset L^{2,1}(\R^+)\subset L^2(\R^+)$ with continous embeddings.
These considerations prove that the last integral is bounded in $L^2(\R^+)$ by
$C\|u\|_X^{2/3}\|v\|_X^{2/3}$.
This in turn implies~\eqref{bil:V}.

Notice, for any $\lambda>0$, 
a rescaled solution $u_\lambda(x,t)=\lambda u(\lambda x,\lambda^2 t)$,
satisfies~\eqref{non-symt} if and only if $u$ does satisfy~\eqref{non-symt}.
Therefore, it is convenient to work with a suitably rescaled datum 
$u_{0,\lambda}=\lambda u_0(\lambda\cdot)$,
in a such way that the smallness assumption $\|u_0\|_2<\delta$
insure that
\[
\|u_{0,\lambda}\|_2+\|u_{0,\lambda}\|_{\dot H^{-1}}< 2\delta<C_0/(4K).
\] 
This is possible taking a large enough $\lambda$, so that
$\|u_{0,\lambda}\|_{\dot H^{-1}}=\lambda^{-1}\|u_0\|_{\dot H^{-1}}=\delta$.

To make the notations lighter in the sequel, we abusively temporary drop
the scaling parameter $\lambda$, and write $u$ instead of $u_\lambda$, 
even though from now on we do work with the rescaled solution.

The global solution $u\in X$ constructed by fixed point 
(that agrees with Leray's solution) satisfy
$u=e^{t\Delta}u_0+B(u,u)$,
with
\[
\|u\|_X\le 2\|u_0\|_X\le 4\delta .
\]
Moreover, for any component of $u$ ($j=1,2$),
\[
u_j^2=(e^{t\Delta}u_{0,j})^2+2e^{t\Delta}u_{0,j}B(u,u)_j+B(u,u)_j^2.
\]
Integrating in space-time we get, for an absolute constant $C>0$,
\[
\begin{split}
\int_0^\infty\!\!\!\int u_j^2
&\ge 
\int_0^\infty\!\!\!\int (e^{s\Delta}u_{0,j})^2
-2\|e^{t\Delta}u_{0}\|_{L^2_{x,t}}\|B(u,u)\|_{L^2_{x,t}}
-\|B(u,u)\|_{L^2_{x,t}}^2\\
&\ge
\int_0^\infty\!\!\!\int (e^{s\Delta}u_{0,j})^2
-C\delta^3-C\delta^4\\
&\ge \int \frac{|\widehat u_0|^2(\xi)}{|\xi|^2}\dd\xi-C\delta^3.
\end{split}
\]
If we now reproduce the same calculation for
$\int_0^\infty\!\!\!\int (u_1^2-u_2^2)$ and for $\int_0^\infty\!\!\! \int 2u_{0,1}u_{0,2}$
we obtain, for another absolute constant $C>0$,
\[
\begin{split}
&\sqrt{\Bigl(\int_0^\infty\!\!\int (u_1^2-u_2^2)\dd y\dd s\Bigr)^2 +\Bigl(\int_0^\infty\!\!\int 2u_1u_2\dd y \dd s\Bigr)^2}\\
&\qquad\ge 
\sqrt{\Bigl(\int \frac{|\widehat u_{0,1}(\xi)|^2-|\widehat u_{0,2}(\xi)|^2}{|\xi|^2}\dd \xi\Bigr)^2 
+\Bigl(\int 2\frac{\widehat u_{0,1}(\xi)\overline{\widehat u_{0,2}}(\xi)}{|\xi|^2}\dd\xi\Bigr)^2}\,\,-C\delta^3\\
&\qquad=
\kappa_0\,\lambda^{-2}-C\delta^3\\
&\qquad=
\kappa_0\,\delta^2\|u_0\|_{\dot H^1}^{-2}-C\delta^3.
\end{split}
\]
The last expression is strictly positive when 
$\delta<\kappa_0/(C\|u_0\|_{\dot H^{-1}}^2)$.
Under this condition and the previous condition $2\delta<C_0/(4K)$, the rescaled solution $u_\lambda$, and hence the non-rescaled solution
$u$ itself, do satisfy~\eqref{non-symt}.
\end{proof}

\begin{remark}
There are examples of (non generic) flows such that $L(t)\equiv0$.
The best known are classical circular flows with radial vorticity, described, \emph{e.g.},  in~\cite{Sch91}. For such flows, $\mathcal{H}(t)$ is not well defined and no hexagonal structure is present. Such flows are somehow trivial, as
the nonlinearity $\P\cdot\nabla(u\otimes u)$ identically vanishes, but very important to describe the large time dynamics of general flows. See \cite{GalW05}.

Following the author (see \cite[Chapt. 25]{Lem02}), we call \emph{symmetric} a 2D flow such that
\begin{itemize}
\item[i)] $(x_1,x_2)\mapsto u_1(x_1,x_2,t)$ is odd with respect to $x_1$ and even with respect to $x_2$.
\item[ii)] $u_1(x_1,x_2,t)=u(x_2,x_1,t)$ for all $x\in \R^2$ and $t\ge0$.
\end{itemize}
Symmetric flows provide another example of non-generic (and non-trivial) solutions 
such that $L(t)\equiv0$.
Let us call ``half-symmetric'' a flow satisfying just one of conditions~i) or ii).
For half-symmetric flows, one in general has $L(t)\not=0$, so that the hexagonal structure $\mathcal{H}(t)$ is present. But
$a-d\equiv0$ or $b\equiv0$: in both cases, one concludes from our previous computations that $\dot{\mathcal{H}}\equiv0$. In other words, for half-symmetric flows the hexagonal structure always remains in a fixed position.

The curious concentration-diffusion effects pointed out in~\cite{Bra09} and also
\cite{FarSY} can be interpretated as follow: there are flows such that $L(t)\not\equiv0$, but such that $L$ has an arbitrarily large number of zeros.
\end{remark}

\section{Proof of Proposition~\ref{prop:dec12} and Proposition~\ref{prop:dec32}}
\label{sec:pro}

The proof of Proposition~\ref{prop:dec12} is carried in two steps. 
In the first one, the solution is proved to belong, for some $T_0>0$ small enough, to a Banach space $X_{p,T_0}$, of functions such that $u$ and $\nabla u$ have a suitable pointwise decay at the spatial infinity, at least for $0<t\le T_0$.
In the second step, the spatial decay for $u$ and $\nabla u$ is proved to persist beyond $T_0$, and to hold also in $[T_0,T]$. This second step is based on an argument of Vigneron~\cite{Vig05}.

Let us consider the  weight functions
\begin{equation}
\phi(x)=(1+|x|)\log(e+|x|)^{1/2}
\quad\text{and}
\quad
\psi(x)=(1+|x|)^{2}\log(e+|x|)^{1/2}.
\end{equation}

For any $T>0$ and  $1\le p\le \infty$, let us set 
\begin{equation*}
\begin{split}
 \|u\|_{X_{p,T}}:=  \esssup_{t\in(0,T)} t^{1/p}\|\phi\, u(t)\|_\infty
    +\esssup_{t\in(0,T)} t^{1/p+1/2}\|\psi\, \nabla u(t)\|_\infty.
  \end{split}
\end{equation*}
So we can define the Banach space $X_{p,T}$ of measurable functions~$u$ 
on $\R^2\times(0,T)$ such that $\|u\|_{X_{p,T}}<\infty$.
We also consider the closed subspace $Y_{p,T}\subset X_{p,T}$ defined as follows:
\[
\begin{split}
Y_{p,T}=\Bigl\{u\in X_{p,T}\colon 
\phi(x)&\esssup_{t\in(0,T)}  t^{1/p}|u(x,t)|\to0 \quad\text{and}\\
& \;\psi(x)\esssup_{t\in(0,T)} t^{1/p+1/2}\,|\nabla u(x,t)|
\to0,\quad\text{as $|x|\to\infty$}\Bigr\}.
\end{split}
\]

\begin{lemma}
\label{lem:bil}
Let  $B(u,v)(t)=-\int_0^t F(t-s)*(u\otimes v)(s)\dd s$ the bilinear term of the Navier--Stokes equations.
For all $T>0$ and all $4<p\le\infty$, we have the estimate
\begin{equation}
\label{bilest}
 \|B(u,v)\|_{X_{p,T}}\le C_p\,T^{1/2-1/p}(1+\sqrt T)\|u\|_{X_{p,T}}\|v\|_{X_{p,T}},
\end{equation}
where $C_p>0$ depends only on $p$.

Moreover, if $u$ and $v$ belong to $Y_{p,T}$ then $B(u,v)$ does also belong to $Y_{p,T}$.
\end{lemma}

\begin{proof}
First of all, we have
\[
\|B(u,v)(t)\|_\infty\le \int_0^t \|F(t-s)\|_1\|u(s)\|_\infty\|v(s)\|_\infty\dd s
\le C_p t^{1-2/p}\|u\|_{X_{p,T}}\|v\|_{X_{p,T}}.
\]
Moreover, for $|x|\ge 2e$,
\[
\begin{split}
|B(u,v)|(x,t)
&\le C|x|^{-3}\int_0^t\int_{|y|\le |x|/2}|u|\,|v|(y,s)\dd y\dd s\\
  &\qquad\qquad+\int_0^t\int_{|y|\ge |x|/2} |F(x-y,t-s)| s^{-2/p}\dd y\dd s \,\,\phi(x)^{-2} \|u\|_{X_{p,T}}\|v\|_{X_{p,T}}\\
&\le C_p\Bigl(t^{1-2/p}|x|^{-3}\log(\log|x|))  +\int_0^t\|F(t-s)\|_1s^{-2/p}\dd s\,\,\phi(x)^{-2}\Bigr)
  \|u\|_{X_{p,T}}\|v\|_{X_{p,T}}\\
&\le C_p(t^{1/2-2/p})(1+\sqrt t)\phi(x)^{-1}\|u\|_{X_{p,T}}\|v\|_{X_{p,T}}\\
\end{split}
\]
Similarly,
\[
\|\nabla B(u,v)(t)\|_\infty\le \int_0^t \|F(t-s)\|_1\|\nabla (u\otimes v)(s)\|_\infty \dd s
\le C_p\, t^{-2/p}\|u\|_{X_{p,T}}\|v\|_{X_{p,T}}.
\]
And, for $|x|\ge e$,
\[
\begin{split}
|\nabla B(u,v)|(x,t)
&\le C|x|^{-3}\int_0^t\int_{|y|\le |x|/2} |\nabla(u\otimes v)|(y,s)\dd y\dd s\\
  &\qquad+C\int_0^t\int_{|y|\ge |x|/2} |F(x-y,t-s)| s^{-1/2-2/p}\dd y\dd s \,\,\phi(x)^{-1}\psi(x)^{-1}
   \|u\|_{X_{p,T}}\|v\|_{X_{p,T}}\\
&\le C_p
 \Bigl(t^{1/2-2/p}|x|^{-3}  +\int_0^t\|F(t-s)\|_1s^{-1/2-2/p}\dd s\,\, \phi(x)^{-1}\psi(x)^{-1}\Bigr)\|u\|_{X_{p,T}}\|v\|_{X_{p,T}}\\
&\le C_p(t^{1/2-2/p}+t^{-2/p})\psi(x)^{-1}\|u\|_{X_{p,T}}\|v\|_{X_{p,T}}\\
\end{split}
\]
Combining the four previous estimates implies~\eqref{bilest}.

If $u$ and $v$ belong to the closed subset $Y_{p,T}$, then going back to the previous estimates one readily see that $B(u,v)\in Y_{p,T}$.

\end{proof}

\begin{lemma}
\label{lem:deli}
Let $1\le p\le \infty$ and $u_0\in L^p(\R^2)$, be such 
$u_0(x)=o(\phi(x)^{-1})$ and $\nabla u_0(x)=o(\psi(x)^{-1})$ as $|x|\to+\infty$.
Then, for all $T>0$, $e^{t\Delta}u_0\in Y_{p,T}$.
\end{lemma}

\begin{proof}
From $u_0\in L^p(\R^2)$, the usual heat kernel estimates gives
\begin{equation}
\label{heato}
\begin{split}
 &\sup_{t>0} t^{1/p}\|e^{t\Delta}u_0\|_\infty+
  \sup_{t>0} t^{1/p+1/2}\|\nabla e^{t\Delta}u_0\|_\infty
  <\infty.\\
\end{split}
\end{equation}
Moreover, there exists $R_0>1$ such that, for all $|x|\ge R_0$ we have 
$|u_0(x)|\le \phi(x)^{-1}$ and 
$|\nabla u_0(x)|\le \psi(x)^{-1}$.
The spatial decay estimates as $|x|\to\infty$ are simple:
splitting the heat integral $\int g_t(x-y)u_0(y)\dd y$ at $y=|x|/2$
and using $g_t(x/2)\int_{|y|\le|x|/2}|u_0(y)|\dd y\le C|x|^{-4}t\,|x|^{2(1-1/p)}$
, one obtains
for all $|x|\ge 2R_0$,
\[
\begin{split}
 |e^{t\Delta}u_0(x)|
 &\le C(t |x|^{-2}+\phi(x)^{-1})\le C(1+t)\phi(x)^{-1}.
\end{split}
\]
with $C>0$ independent on $t$.
Hence,
\[
\sup_{t\in(0,T)} t^{1/p}\|\phi\, e^{t\Delta}u_0\|_\infty\le C(1+T^{1+1/p}).
\]

Next, let $\chi$ be a cut-off function equal to $1$ for $|x|\le R_0$, vanishing for $|x|\ge 2R_0$.
We have $\nabla e^{t\Delta}u_0=(\nabla g_t)*(\chi u_0)+g_t*\nabla[(1-\chi)u_0]$.
Hence, for $|x|\ge4R_0$,
\[
\begin{split}
 |\nabla e^{t\Delta}u_0|(x)
 &\le C|x|^{-3}\int_{|y|\le 2R_0}|u_0|+
 g_t(x/2)\int_{R_0\le |y|\le |x|/2}|\nabla [(1-\chi)u_0]|
 +\esssup_{|y|\ge |x|/2} |\nabla u_0|(y)\\
 &\le C(1+\sqrt t)\psi(x)^{-1}.
\end{split}
\]
Therefore,
\[
\sup_{t\in(0,T)} t^{1/p+1/2}\,\|\psi\, \nabla e^{t\Delta}u_0\|_\infty\le C(1+T^{1+1/p}).
\]
In fact, using that $u_0(x)=o(\phi(x)^{-1})$ and $\nabla u_0(x)=o(\psi(x)^{-1})$ as $|x|\to\infty$, allow us to reinforce previous conclusion into 
$\phi_1(x)\sup_{t\in(0,T)}|e^{t\Delta}u_0|(x)\to0$
and  $\phi_2(x)\sup_{t\in(0,T)}|e^{t\Delta}u_0|(x)\to0$, getting $e^{t\Delta}u_0\in Y_{p,T}$.

\end{proof}

\begin{proof}[Proof of Proposition~\ref{prop:dec12}]\hfill\mbox{}
We first make use of  the spatial decay assumption for $u_0, \nabla u_0$, of the condition $u_0\in L^p(\R^2)$, with $4<p\le\infty$ and the divergence-free condition
on $u_0$. 
Observe that norm of the bilinear operator
of $B\colon X_{p,T}\times X_{p,T}\to X_{p,T}$ goes to zero as $T\to0$, as we checked in establishing~\eqref{bilest}.
If we choose $T_0>0$ small enough, then applying the the standard fixed point argument in $Y_{p,T_0}$ we get from the two previous Lemmas the existence of a local-in-time solution $u\in Y_{p,T_0}$ of the Navier--Stokes equations, written in its integral form, $u=e^{t\Delta}u_0+B(u,u)$.
This solution is obtained as the limit $u=\lim_{k\to\infty}u_k$ in the 
$X_{p,T_0}$-norm, where, accordingly with the usual iteration scheme, 
$u_1:=e^{t\Delta}u_0$ and
$u_{k+1}=u_1+B(u_k,u_k)$ for $k=1,2,\ldots$

In fact, $u_0$ does also belong to $L^2_\sigma(\R^2)$, and the above iteration scheme is known to converge also in $L^4([0,T_0],\dot H^{1/2}(\R^2))$ by classical Fujita and Kato's result. Therefore the solution $u$ agrees with Leray's solution in such time interval. (See \cite{BahCD11}).
%
%
%
But Leray's solution is defined beyond $T_0$ and is such that,
for all $T>T_0$,
\begin{equation}
\label{katin}
\sup_{t\in [T_0,T]} \|u(t)\|_\infty< \infty.
\end{equation}
See, \emph{e.g.}, \cites{SawT07, Zel13} for fine $L^\infty$-estimates of 2D Navier-Stokes flows
valid also in the more general settings of infinite energy solutions.
We now work on $[T_0,T]$, where $T>T_0$ is arbitrary.
It will be convenient to consider the new initial datum
\[
\tilde u_0(x)=u(x,T_0).
\]
From the fact that $u\in Y_{p,T_0}$ we infer that
\begin{equation}
\label{dectil}
\begin{aligned}
&\phi\,\tilde u_0\in L^\infty(\R^2)\\
&\psi\,\nabla \tilde u_0\in L^\infty(\R^2)
\end{aligned}
\qquad
\text{and}
\qquad
\begin{aligned}
&\phi(x)\tilde u_0(x)\to0\\
&\psi\nabla \tilde u_0(x)\to0
\end{aligned}
\qquad\text{as $|x|\to\infty$}.
\end{equation}

We now argue as in Vigneron's paper~\cite{Vig05} to
deduce, from~\eqref{katin} and~\eqref{dectil},
that the spatial decay of $\tilde u_0$ and $\nabla \tilde u_0$ is preserved by the
flow, in the whole interval $[T_0,T]$.
To this purpose, let us introduce, for $a>0$,
\[
\varphi_a(x)=(1+|x|)^{a}.
\]
Observe that $\varphi_a$ is submultiplicative for all $a>0$, hence
$\varphi_a(x)\le \varphi_a(x-y)\varphi_a(y)$.
For the moment, we take
\[
a=2/3.
\]
Next we will improve the decay rates for $u$ by bootstrapping.
First of all, we have, for all $T_0\le s<t\le T$,
\begin{equation}
\label{nss}
u(t)=e^{(t-s)\Delta} u(\cdot,s)-\int_{s}^tF(t-\tau)*(u\otimes u)(\tau)\dd \tau.
\end{equation}
By the scaling relations and the decay of the kernel $F$ 
we have, for any $T>0$,
\begin{equation*}
\begin{cases}
\sup_{t\in [0,T]} t^{1/2} \| \varphi_a F(\cdot,t)\|_1<\infty\\
\sup_{t\in [0,T]} t \|\varphi_a\nabla F(\cdot,t)\|_1<\infty,
\end{cases}
\end{equation*}
(we cannot take here $a=1$ because $\|\varphi_1F(t)\|_1=\infty$. On the other hand,
any choice of $a\in(0,1)$ would do).
The following linear estimates hold:
\begin{equation}
\label{linper}
\sup_{t\in[0,T]} \|\varphi_a\,e^{t\Delta}v\|_\infty
+\sup_{t\in[0,T]} \|\varphi_a\,\nabla e^{t\Delta}v\|_\infty
\le A(\|\varphi_a v\|_\infty+\|\varphi_a\nabla v\|_\infty),
\end{equation}
with $A=A(T)$.
We have, for all  $T_0\le \tau\le t\le T$, and for some constant $B=B(T)$ independent on $\tau$ and $t$, and $T_0$,
\begin{equation}
\label{nonlinper}
\begin{split}
|\varphi_a(x) F(t-\tau)*(u\otimes u)(x,\tau)|
&\le \int \varphi_a(x-y) | F(x-y,t-\tau)| \,\varphi_a(y)\,|u(y,\tau)| \,|u(y,\tau)|\dd y\\
&\le \|\varphi_a F(t-\tau)\|_1\|\varphi_a u(\tau)\|_\infty\|u(\tau)\|_\infty\\
&\le B(t-\tau)^{-1/2}\|\varphi_a u(\tau)\|_\infty\|u(\tau)\|_\infty .
\end{split} 
\end{equation}
In the same way,
\begin{equation}
\label{nonlinper2}
\begin{split}
\Bigl|\varphi_a(x)\nabla [F(t-\tau)*(u\otimes u)](x,\tau)\Bigr|
&\le \int  \varphi_a(x-y) \,| F(x-y,t-\tau)|\,\varphi_a(y)\,|\nabla u(y,\tau)|\,|u(y,\tau)|\dd y\\
&\le \|\varphi_a F(t-\tau)\|_1\|\varphi_a\nabla u(\tau)\|_\infty\|u(\tau)\|_\infty\\
&\le B(t-\tau)^{-1/2}\|\varphi_a\nabla u(\tau)\|_\infty\|u(\tau)\|_\infty .
\end{split} 
\end{equation}
Combining the two latter estimates with~\eqref{linper}, we get from equation~\eqref{nss},
for all $T_0\le s\le t\le T$,
\begin{equation}
\label{soppo}
\begin{split}
\|\varphi_au(t)&\|_\infty+\|\varphi_a\nabla u(t)\|_\infty\\
&\le A(\|\varphi_a u(s)\|_\infty+\|\varphi_a\nabla u(s)\|_\infty)\\
&\qquad+ 2B(t-s)^{1/2}\sup_{\tau\in[s,t]}\|u(\tau)\|_\infty \, 
 \Bigl(
 \sup_{\tau\in[s,t]}\|\varphi_a u(\tau)\|_\infty+\sup_{\tau\in[s,t]}\|\varphi_a \nabla u(\tau)\|_\infty\Bigr).
\end{split}
\end{equation}

We may assume $u\not\equiv0$ on $[T_0,T]$. Starting with $T_0$, we construct a strictly increasing sequence of times $(T_i)_{i\ge0}$ such that, for $i\ge0$,
\[
2B (T_{i+1}-T_i)^{1/2}\sup_{\tau\in [T_0,T]}\|u(\tau)\|_\infty=1/2.
\]
Let $N\in\N$ be such that $T_N\le T<T_{N+1}$.
We thus have
\[
N\le (T-T_0)(4B\sup_{\tau\in [T_0,T]}\|u(\tau)\|_\infty)^2< N+1.
\]
For $i\le N$, consider the interval $\Delta_i=[T_0,T]\cap[T_i,T_{i+1}]$,
and set 
\[
M_i=\sup_{t\in\Delta_i}\|\varphi_a u(\tau)\|_\infty
+\sup_{t\in\Delta_i}\|\varphi_a \nabla u(\tau)\|_\infty.
\]
Applying \eqref{soppo} with $s=T_0$ and $T_0\le t\le T_1$
we get
\begin{equation}
\label{estiM}
M_0\le 2A(\|\varphi_{a}\tilde u_0\|_\infty+\|\varphi_{a}\nabla \tilde u_0\|_\infty).
\end{equation}
In the same way, working on $\Delta_i$, for $1\le i\le N$ we get
\begin{equation}
\label{estiM2}
M_i\le 2AM_{i-1}.
\end{equation}
Therefore,
\begin{equation}
\label{metti}
\sup_{t\in[T_0,T]}\|\varphi_a u(t)\|_\infty
+ \sup_{t\in[T_0,T]}\|\varphi_a \nabla u(t)\|_\infty
\le (2A)^{N+1}(\|\varphi_{a}\tilde u_0\|_\infty+\|\varphi_{a}\nabla \tilde u_0\|_\infty).
\end{equation}
But
\[
(2A)^{N+1}= (2A)\exp(N\log(2A))
\le A_1\exp\bigl(A_2\sup_{t\in [T_0,T]}\|u(\tau)\|_\infty^2\bigr),
\]
where $A_1$ and $A_2$ depend only on~$T$, and are locally bounded functions of~$T$.
We then conclude that
\begin{equation}
\label{mettip}
\sup_{t\in[T_0,T]}\|\varphi_a u(t)\|_\infty
+ \sup_{t\in[T_0,T]}\|\varphi_a \nabla u(t)\|_\infty
\le 
A_1\exp\bigl(A_2\sup_{t\in [T_0,T]}\|u(\tau)\|_\infty^2\bigr)
(\|\varphi_{a}\tilde u_0\|_\infty+\|\varphi_{a}\nabla \tilde u_0\|_\infty).
\end{equation}

We now finish the proof with some bootstrapping on the spatial decay rate.
From the above estimates we get the provisory spatial decay
\[
\sup_{t\in[T_0,T]}|u(x,t)|+\sup_{t\in[T_0,T]}|\nabla u(x,t)|\le C(\varphi_{2/3}(x))^{-1},
\]
where $C$ depends only on $T$ and on the initial data $u_0$.
But
\[
u(x,t)=e^{(t-T_0)\Delta}\tilde u_0(x)-\int_{T_0}^t F(t-s)*(u\otimes u)(x,s)\dd s.
\]
For the quadratic term we easily get from our provisory estimate
\[
\sup_{t\in[T_0,T]}\int_{T_0}^t |F(t-s)|*|u\otimes u|(x,s)\dd s\le C(\varphi_{4/3}(x))^{-1}=o(\phi(x)^{-1})
\qquad\text{as $|x|\to+\infty$}.
\]
For the linear term, recalling~\eqref{dectil} we have
$\phi_1\,e^{(t-T_0)\Delta}\tilde u_0 \in L^\infty(\R^2)$ and 
  $\sup_{t\in[0,T]}  |e^{(t-T_0)\Delta}\tilde u_0|(x)=o(\phi(x)^{-1})$
 as $|x|\to+\infty$.
Hence,
\[
\phi\sup_{t\in[T_0,T]} |u(\cdot,t)| \in L^\infty(\R^2)\qquad\text{and}\qquad
\phi(x)\sup_{t\in[T_0,T]}|u(x,t)|\to0, \qquad\text{as $|x|\to\infty$}.
\]  
For the gradient estimates, after a similar bootstrapping procedure (but with a few more
iterations) we get
\[
\psi\sup_{t\in[T_0,T]} |\nabla u(\cdot,t)| \in L^\infty(\R^2)\qquad\text{and}\qquad
\psi(x)\sup_{t\in[T_0,T]}|\nabla u(x,t)|\to0, \qquad\text{as $|x|\to\infty$}.
\] 

This concludes the proof of Proposition~\ref{prop:dec12}. 
\end{proof}

Before proving Proposition~\ref{prop:dec32}, we set
$\varphi:=\varphi_{3/2}$, \emph{i.e.\/}, according to our previous notation,
\[
\varphi(x)=(1+|x|)^{3/2}.
\]
Moreover, for $p>2$ we set
\[
\|u\|_{Z_{p,T}}:=
\esssup_{t\in(0,T)}t^{1/4}\|u(t)\|_4+\esssup_{t\in(0,T)} t^{1/p}\|\varphi\, u\|_\infty.
\]
We denote by $Z_{p,T}$ the Banach space of measurable functions $u$ on $\R^2\times(0,T)$
such that $\|u\|_{Z_{p,T}}<\infty$ and by $W_{p,T}$ the closed subspace of~$Z_{p,T}$,
\[
W_{p,T}=\{u\in Z_{p,T}\colon \lim_{|x|\to+\infty} \varphi(x)\esssup_{t\in(0,T)}
t^{1/p}|u(x,t)|=0\}.
\]
We equip $W_T$ with the $Z_T$-norm.

\begin{proof}[Proof of~Proposition~\ref{prop:dec32} ]
The only important change with respect to the proof of~Proposition~\ref{prop:dec32}, 
is the following bilinear estimate, that conveniently replaces
Lemma~\ref{lem:bil}
\[
\|B(u,v)\|_{Z_{p,T}}\le CT^{1/2-1/p}(1+\sqrt T)\|u\|_{Z_{p,T}}\|v\|_{Z_{p,T}},
\]
valid for all $T>0$ and $p>2$, where $C>0$ depends only on~$p$.

This is elementary:
first of all,
\[
\|B(u,v)(t)\|_4\le C\int_0^t\|F(t-s)\|_{1}\|u(s)\|_4\|v(s)\|_\infty\dd s
\le Ct^{1/4-1/p}\|u\|_{Z_{p,T}}\|v\|_{Z_{p,T}}.
\]
On the other hand,
\[
\|B(u,v)(t)\|_\infty\le C\int_0^t\|F(t-s)\|_{1}\|u(s)\|_\infty\|v(s)\|_\infty\dd s
\le Ct^{1/2-2/p}\|u\|_{Z_{p,T}}\|v\|_{Z_{p,T}}.
\]
Moreover, for $|x|\ge 1$, after splitting as usual the integrals defining
$B(u,v)(x,t)$ in the regions $\{|y|\le |x|/2\}$ and $\{|y|\ge |x|/2\}$,
we obtain
\[
\begin{split}
|B(u,v)|(x,t)
&\le C|x|^{-3}\int_0^t\int_{|y|\le |x|/2}|u|\,|v|(y,s)\dd y\dd s
 + C\|u\|_{Z_{p,T}}\|v\|_{Z_{p,T}}\int_0^t\|F(t-s)\|_1 s^{-2/p}\dd s\,\,\varphi(x)^{-2}\\
&\le C\Bigl(t^{1-2/p}|x|^{-3}+t^{1/2-2/p}\varphi(x)^{-2}\Bigr)\|u\|_{Z_{p,T}}\|v\|_{Z_{p,T}}\\
&\le C(1+\sqrt t)t^{1/2-2/p}\varphi(x)^{-1} \|u\|_{Z_{p,T}}\|v\|_{Z_{p,T}}.
\end{split}
\]
This establishes the required bilinear estimate.

Notice that if the two functions (or at least one of them) $u$ and $v$ belong more precisely to $W_{p,T}$, then the last estimate ensures that $B(u,v)\in W_{p,T}$.

On the other hand, if $u_0\in L^2_\sigma(\R^2)$ and $u_0(x)=o(|x|^{-3/2})$ as 
$|x|\to+\infty$, then one easily checks via standard heat kernel estimates that
$e^{t\Delta}u_0\in W_{p,T}$, for all $T>0$.
Therefore, choosing a small enough $T_0>0$ the usual fixed point argument
applies in $W_{p,T_0}$. Hence, we get the existence of a solution
$u\in W_{p,T_0}$. This solution agrees with Leray's solution on $[0,T_0]$. By a continuation argument, similar to the one we did in Proposition~\ref{prop:dec12}, we  finally conclude that $u\in W_{p,T}$, for all $T>0$.

\end{proof}

%
%

\section{Acknowledgements}
The author would like to thank Prof. Tsuyoshi Yoneda for some interesting discussions that motivated the present study.

\end{document}